\documentclass[11pt]{article}
\usepackage{amsthm}
\usepackage{indentfirst}
\usepackage{subfigure}
\usepackage{enumitem}
\usepackage{amssymb}
\usepackage{enumerate}
\usepackage{enumitem}
\usepackage{amsmath,amsthm,marvosym,wasysym,mathdots}
\usepackage{bbm}
\usepackage[round]{natbib}
\usepackage{color}
\usepackage{setspace}
\usepackage{epsfig}
\usepackage{hyperref}
\usepackage{amsmath}
\usepackage{graphicx}
\usepackage{ae,aecompl}
\usepackage[multiple]{footmisc}

\hypersetup{%
  colorlinks=false,
  pdfborder = {0 0 0.5 [3 0]}
}

\usepackage{multirow}
\usepackage[left=1in,right=1in,top=1in,bottom=1in]{geometry}

\usepackage[margin=1cm]{caption}

\newcommand{\ee}{\mathbf{e}}

\newtheorem{theorem}{Theorem}
\newtheorem*{theorem*}{Theorem}

\newtheorem{corollary}[theorem]{Corollary}

\newtheorem{lemma}[theorem]{Lemma}

\newtheorem{proposition}[theorem]{Proposition}
\newtheorem{remark}[theorem]{Remark}

\def\P{{\mathbb P}} 
\newcommand{\EE}{{\mathord{I\kern -.33em E}}}

\def\1{1{\hskip -3.3 pt}\hbox{I}}

\newcommand{\ind}{1\hspace{-2.1mm}{1}} 

\providecommand{\varitem}{}


\numberwithin{equation}{section}
\numberwithin{theorem}{section}

\begin{document}
\title{Timing Options for a Startup with Early Termination and Competition Risks}
\author{Tim Leung\thanks{ \small{Department of Applied Mathematics, University of Washington, Seattle, WA 98195. E-mail:
\mbox{timleung@uw.edu}.} Corresponding author. } \and Zongxi Li\thanks{\small{Operations Research \& Financial Engineering Department, Princeton University,
Princeton, NJ 08544. E-mail:
\mbox{zongxil@princeton.edu}. } } }\date{\today} \maketitle

\begin{abstract}
This paper analyzes the timing options embedded in a startup firm, and the associated market entry and exit timing decisions under the exogenous risks of early termination and competitor's entry. Our valuation approach leads to the analytical study of a non-standard perpetual American installment option nested with an optimal sequential stopping problem. Explicit formulas are derived for the firm's value functions. Analytically and numerically, we show that early termination risk leads to earlier voluntary entry or exit, and the threat of competition has a non-trivial effect on the firm's entry and abandonment strategies.

\end{abstract}



  {\textbf{Mathematics Subject Classification (2010):}\, 60G40, 62L15, 91G20,  91G80 \\

\newpage
\section{Introduction}

Over the past two decades, there have been a number of highly successful startup companies. Nevertheless, in general firms in their early stages require significant capital investment for research and development (R\&D) and  face tremendous amount risk that may force them to abandon the projects. Many  startups may fail before they even sell a single unit of products. The risk of early termination  may be due to demand uncertainty,  technological challenges,   regulatory changes, and other causes.

Another major risk factor for a startup is the emergence of a competitor.  This is commonly seen in markets with very few major  players, for example,   Lyft's entry to  New York City to compete with Uber,\footnote{http://www.wsj.com/articles/lyft-revs-up-in-new-york-city-1448038672} and the launch of Apple music to rival Pandora and Spotify.\footnote{http://www.wsj.com/articles/apple-to-announce-new-music-services-1433183201}  In effect, new competition may reduce the existing firm's market share and thus revenue stream, as noted by \cite{Petersen2001}.  In turn, this influences how long the firm can stay in the market.

In this paper, we propose a theoretical real option approach to better understand  a startup firm's strategies to voluntarily abandon a project or enter the market in face of  early termination  and  competition risks.  The arrivals of early termination and competition are modeled by exogenous independent exponential  random variables. We analyze their combined impact on the firm's optimal timing to enter or exit.   Our valuation approach is a non-standard perpetual American installment option nested with an optimal sequential stopping problem.

Our main results are the analytic solutions for the firm's sequential  timing problems.  We also investigate the impact of the  early termination risk and the competitor's potential entry on the firm's value and the associated  timing to enter and exit the market. Among our findings, the firm has a lower  entry threshold  and higher  cancellation threshold when   the early termination risk rises. On the other hand, competition risk may increase or decrease  the  firm's abandonment level, depending on the impact on the new cash flow. We provide the necessary and sufficient condition for both cases.  Numerical results are provided to illustrate these effects.

In the literature, \cite{Pennings1997} study   the empirical option value of an R\&D   project, and model the arrivals of new information that impact cash flows  by exponential 	random variables.    \citet*{Lukas2016} propose an entrepreneurial venture financing model without competition risk by combining compound option pricing with sequential non-cooperative contracting.   \cite{Kort2015} apply a game-theoretic approach to study  the incumbent's over or under investment  problem accounting for  the threat of entry.  Also, \cite{Restrepo2015} study a real option approach for    sequential investments    in the presence of expropriation risk.     \cite{Davis2004} study  the valuation of a  venture capital  as an installment option,  while  \cite{Ciurlia2009} and \cite{Kimura2009} study the pricing of  American installment put and call options.  Our model can be viewed as  a  perpetual American-style installment compound option with three stopping times. The incubation period in our model is similar to the time-to-build in infrastructure investments, which also involves  a timing option to start operation; see \citet*{Eric2015}.   \cite{Kwon2010} studies the firm's decision to discard or invest in   an aging technology  with a declining profit stream with demand uncertainty. The theory of optimal stopping is applied to study project management with uncertain completion in \cite{ChiChan}, and to develop strategies for  drug discovery in \cite{Zhao2009754}, among others. Sequential  stopping problems  also arise in other applications, such as participating a government subsidized program  (\cite{kuno2013}), and  trading under mean reversion (\cite{LeungLibook,LeungLi2015OU}).

This paper is organized as follows. In Section \ref{sect-formulate}, we introduce the firm's investment timing problems, and  present our formulation.  Then, we present  the solutions of the optimal abandonment timing problem in Section \ref{sect-exit}, and the optimal entry timing problem in Section \ref{sect-entry}.  In Section \ref{sect-sen}, we examine the effects of the early termination risk and competitor's entry on the firm's strategies. Section \ref{sect-conclude} concludes. All proofs are included   in the Appendix.

\section{Problem Formulation}\label{sect-formulate}
In the background, we fix a complete probability space $(\Omega, \cal{F}, \mathbb{P})$. The firm's  valuation is conducted under the  historical probability measure  $\P$, with a subjective discount rate  $\rho>0$, as in  \cite{McDonald86}. We consider a startup firm  that can enter a targeted market and generate  a stochastic cash flow  until the firm's voluntary  abandonment time.  We assume that the firm  acts as a price taker in this market.  The cash flow is driven by a  stochastic factor $X$, satisfying
\begin{equation}
dX_t=\mu X_t\, dt + \sigma X_t\, dB_{t},
\end{equation}
with drift $\mu>0$ and volatility $\sigma>0$. For instance, we can interpret $X$ as the prevailing market price of the goods/services sold by the firm.  Let  $\mathbb{F} = (\mathcal{F}_t)_{t\ge 0}$ be the filtration generated by $X$, and  $\cal{T}$ be the set of all stopping times with respect to $\mathbb{F}$.

The firm seeks to maximize its net present value (NPV) by selecting the best time to enter the market. Prior to the entry, as depicted by Figure \ref{fig:sample1},  the firm can cancel the project at time $\tau_{c}$ (bottom path), or is forced to abort the project due to early termination risk at time $\zeta_{1}$ (middle path) during this incubation period. If either event occurs, the firm stops operation and generates no future cash flow. In the third scenario (top path), the firm avoids early termination and opts to enter  the market at time $\tau_{e}$.

Figure  \ref{fig:sample2} illustrates the scenarios after market entry. Given that the firm has entered the market, the firm  can abandon  the project before or after the competitor's arrival   time $\zeta_{2}$.  In these two possible situations, the firm's abandonment decision with or without competition, represented by the thresholds $\widetilde{a}$ and $a$ respectively,  can be different.  For instance, the firm may have a higher abandonment threshold   if there is competition in the market, i.e. $\widetilde{a}>a$ (top and bottom paths). As such, the  firm will stay in   the market even when  the   price $X$  is below $\widetilde{a}$ but above $a$   before the competitor arrives.  In this case, as soon as the competitor enters suddenly at time $\zeta_2$,  the firm's abandonment level  switches to the higher level $\widetilde{a}$ above the current value of $X$, forcing the firm to abandon immediately (middle path).  The random times involved in our model  are summarized in Table \ref{table:random times}.\\

\begin{table}[h]
\centering
\begin{tabular}{|l|l|l}
\hline
Random times & Description  \\
\hline\hline
$\tau_{c}$ & Firm's cancellation time to forgo market entry \\
$\tau_{e}$ & Firm's entry time \\
$\tau_{a}$ & Firm's abandonment time after market entry without new competition \\
$\widetilde{\tau}_{a}$ &  Firm's abandonment time after competitor's entry \\
$\zeta_{1}$ & Exogenous early termination time during the incubation period   \\
$\zeta_{2}$ & Competitor's exogenous arrival time after the firm's entry   \\
\hline
\end{tabular}
\caption{\small{Summary of the random times in our model. The early termination time $\zeta_1$ and the competitor's arrival time $\zeta_2$ are exogenous and assumed to be independent exponential random variables.}}
\label{table:random times}
\end{table}

\clearpage

\begin{figure}[h]
\centering
\subfigure[Pre-entry period.]{\label{fig:sample1} \includegraphics[width=3.1in,trim=0.2cm 0.4cm 1cm 0.8cm, clip=true]{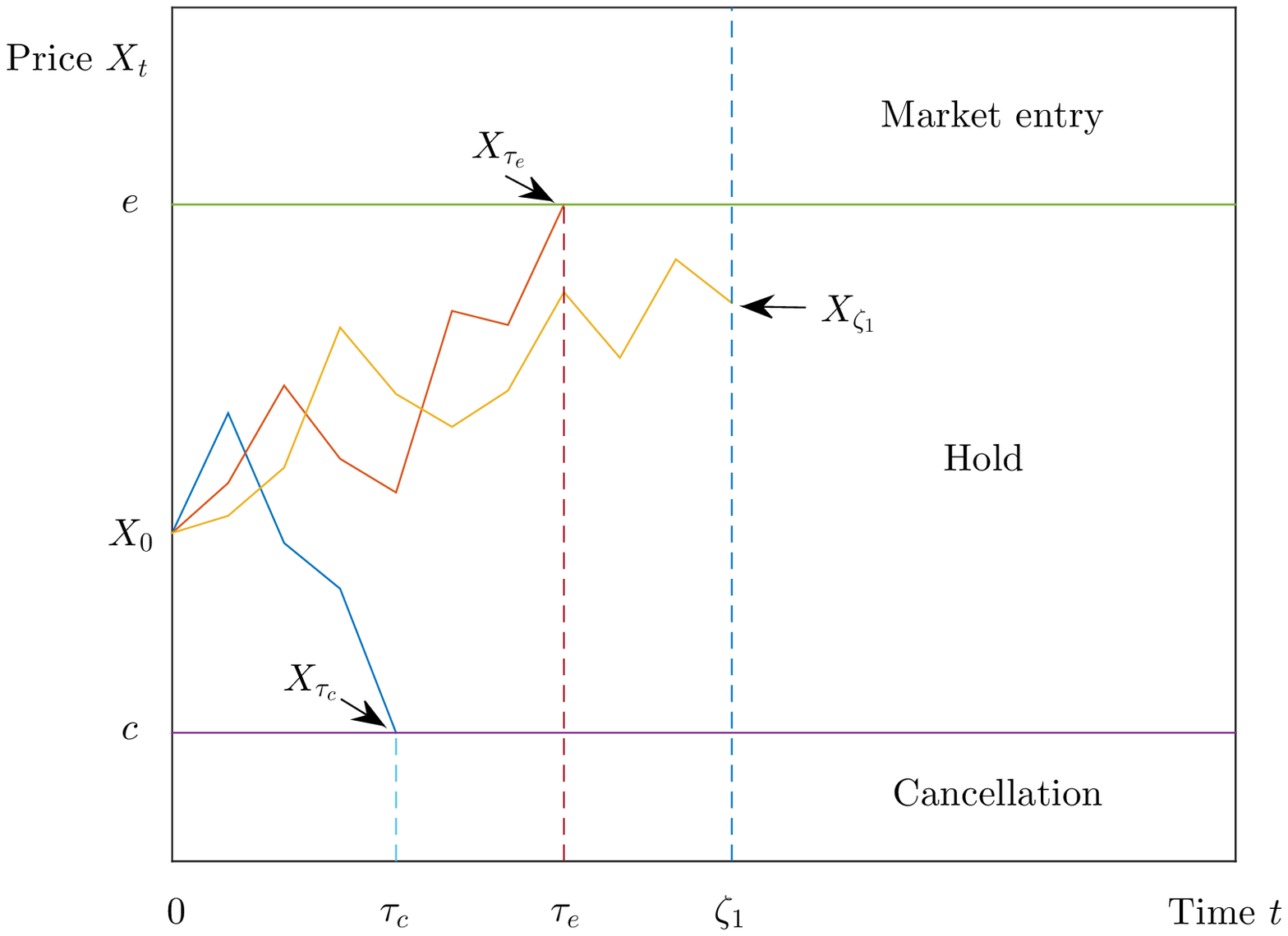} }
\subfigure[Post-entry period.]{\label{fig:sample2} \includegraphics[width=3.1in,trim=0.2cm 0.4cm 1cm 0.8cm, clip=true]{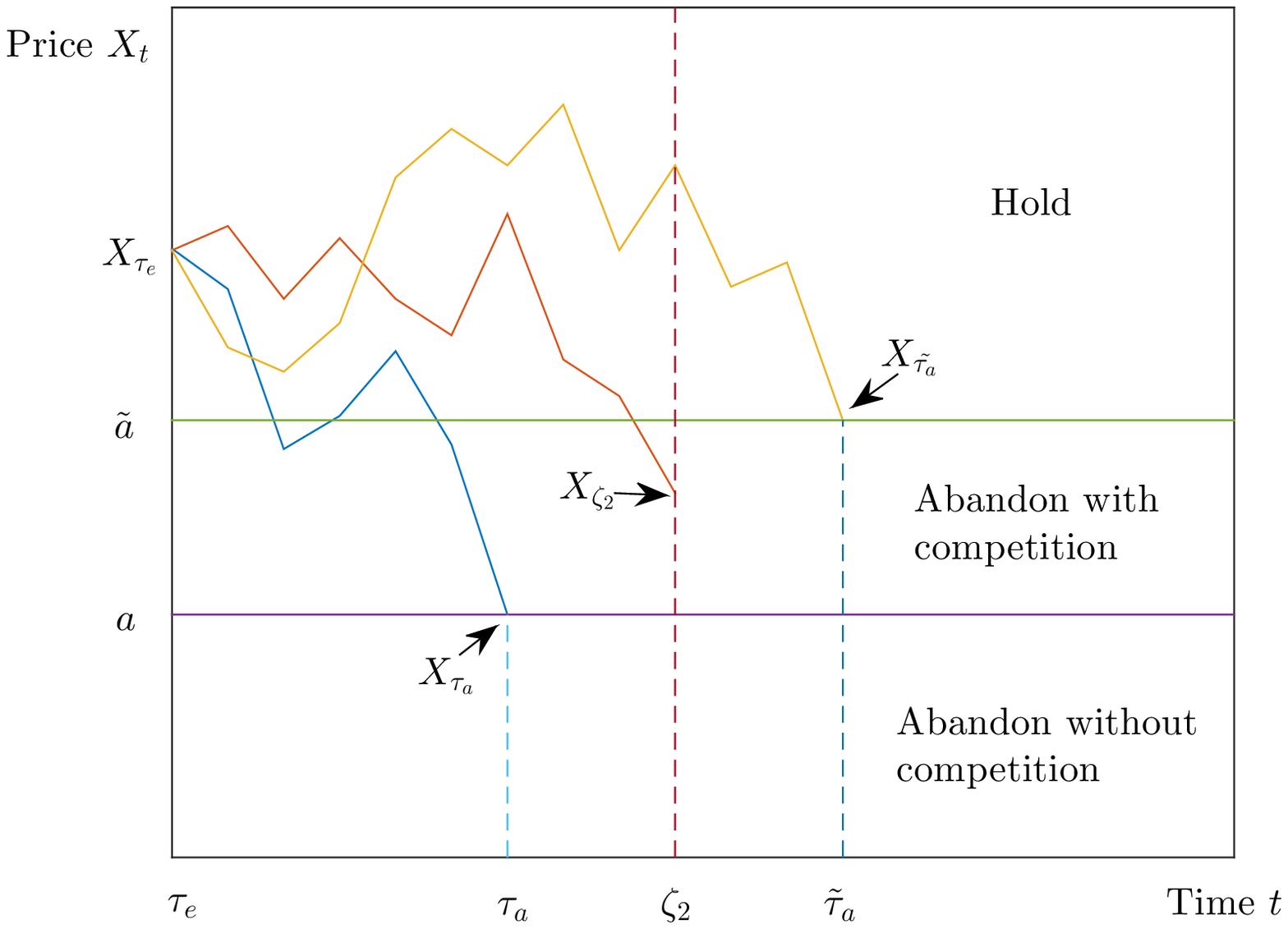} }
\caption{All possible scenarios for the firm in the pre-entry and post-entry periods.}
\label{fig:All possible scenarios}
\end{figure}

To formulate the firm's timing problems, we first  consider the scenario with a competition in the market post-entry period. In this case, we specify   the   profit stream by $(g\left(X_t\right))_{t\ge 0}$. Our model assumes a  linear function: $g(x)=\alpha x-\beta$ for all $x>0$, where $0<\alpha\le 1$, $\beta>0$. We can interpret the fraction $\alpha$ as a reduction in revenue, and $\beta$ as the fixed   cost when the firm faces new competition. In addition, let $\widetilde{\tau}_{a}$ be the abandonment time after the competitor's entry. To maximize its net present value (NPV), the firm solves the optimal stopping problem
\begin{equation}
\label{expression_tildeV}
\widetilde{V}(x) = \sup_{\substack{\widetilde{\tau}_{a}\in\cal{T}}}\mathbb{E}_{x}\left[\int_{0}^{\widetilde{\tau}_{a}}\ee^{-\rho t}g\left(X_{t}\right)\,dt \right],
\end{equation}
where  $\mathbb{E}_{x}[\,  {\cdot} \, ]\equiv\mathbb{E}[\,  {\cdot} \, |X_0=x]$.

When the firm first enters the market,  it   generates a profit stream $(f\left(X_t\right))_{t\ge 0}$, where we assume   $f(x)=x-K$ for all $x>0$,  where $K>0$ represents the fixed cost. Then the firm  operates till the abandonment time  $\tau_{a}$ when facing no new competition.   If the competitor arrives at $\zeta_{2}$ (before the firm's abandonment), then   the firm's NPV is exactly $\widetilde{V}(x)$ (see \eqref{expression_tildeV}). In other words, the firm will continue to generate the profit stream $g$ from time $\zeta_2$ on, after having accumulated  the profit stream $f$  up to  $\zeta_{2}$. We assume that $\zeta_{2}$ is an exponential random variable with parameter $\lambda_{2}$, independent of price process $(X_t)_{t\ge 0}$. Therefore, before new competition arrives, the firm's maximized NPV is
\begin{align}
V(x) &= \sup_{\substack{\tau_{a}\in\cal{T}}}\mathbb{E}_{x}\left[\int_{0}^{\tau_{a}\wedge\zeta_{2}}\ee^{-\rho t}f\left(X_{t}\right)\,dt + \ind_{\{\tau_{a}>\zeta_{2}\}}\ee^{-\rho\zeta_{2}}\widetilde{V}(X_{\zeta_{2}}) \,\right] \nonumber  \\
&=\sup_{\substack{\tau_{a}\in\cal{T}}}\mathbb{E}_{x}\left[\int_{0}^{\tau_{a}}\ee^{-(\rho+\lambda_{2})t}\left(f\left(X_{t}\right)+\lambda_{2}\widetilde{V}
 \left(X_{t}\right)\right) \,dt \,\right],\label{expression_V}
\end{align}which follows from the distribution of $\zeta_2$ and law of iterated expectations. The   derivation is provided in the Appendix \ref{app1}. As we can see, the  value function in \eqref{expression_tildeV} becomes an input to the firm's value function $V$ in \eqref{expression_V}.

During the incubation period,   the firm has to pay the operating cost $c(X_{t})$ over time. For simplicity, we let  $c(x)=ax+b$ for all $x>0$, with constants $a,b>0$. The firm may  choose to cancel the project at time $\tau_{c}$, but could  end up terminating it early at the exogenous   time $\zeta_{1}$, which is an exponential random variable with parameter $\lambda_{1}$, independent of price process $(X_t)_{t\ge 0}$ and $\zeta_{2}$. If the firm avoids cancellation and termination, then it will   enter the market at time $\tau_{e}$. Therefore, the firm's pre-entry value function is given by \begin{align}
\psi(x)
&=\sup_{\substack{\tau_{e},\tau_{c}\in\cal{T}}}\mathbb{E}_{x}\left[- \int_{0}^{\tau_{e}\wedge\tau_{c}\wedge\zeta_{1}}\ee^{-\rho t}c(X_{t})\,dt
+\ind_{\{\tau_{c}\wedge\zeta_{1}>\tau_{e}\}}\ee^{-\rho\tau_{e}}V(X_{\tau_{e}})\, \right]  \notag\\
&= \sup_{\substack{\tau_{e},\tau_{c}\in\cal{T}}}\mathbb{E}_{x}\left[- \int_{0}^{\tau_{e}\wedge\tau_{c}}\ee^{-(\rho+\lambda_{1})t}c(X_{t})\,dt
+\ind_{\{\tau_{c}>\tau_{e}\}}\ee^{-(\rho+\lambda_{1})\tau_{e}}V(X_{\tau_{e}})\, \right].\label{expression_psi}
\end{align}
The last step \eqref{expression_psi} is derived similarly to  \eqref{expression_V}; see Appendix \ref{app1}.

In summary, we seek to solve for the value functions $\{\widetilde{V}, V, \psi\}$ and the associated optimal timing strategies $\{\widetilde{\tau}_a^*, \tau_a^*, \tau_c^*, \tau_e^*\}$. As seen in \eqref{expression_tildeV}-\eqref{expression_psi}, the firm's post-entry strategy and new competition can affect its pre-entry decisions.

\section{Optimal Abandonment Timing Problem}\label{sect-exit}

We first analyze the firm's exit problem represented by $\widetilde{V}(x)$   in \eqref{expression_tildeV}. If $\rho\leq\mu$, then we have
\begin{equation}
\widetilde{V}(x) \geq \mathbb{E}_{x}\left[\int_{0}^{+\infty}\ee^{-\rho t}g\left(X_{t}\right)\,dt \right] \geq \alpha\int_{0}^{+\infty}\ee^{-\rho t}\mathbb{E}_{x}\left[X_{t}\right]\,dt - \frac{\beta}{\rho} = +\infty.
\end{equation}
The last equality holds due to   $\mathbb{E}_{x}(X_{t})=x\ee^{\mu t}$ and $\mu \ge \rho$.  As a result,   it is optimal for the firm to never exit the market if  $\rho\leq\mu$, with an infinite NPV.

Therefore, in  the rest of the paper, we consider the non-trivial case $\rho>\mu$.   To facilitate our presentation, we define the constants \begin{equation}
\label{general}
h_{i}(\lambda)=\frac{1}{\sigma^{2}}\left(\frac{\sigma^{2}}{2}-\mu+(-1)^{i}\,\sqrt[]{(\frac{\sigma^{2}}{2}-\mu)^{2}+2\sigma^{2}(\rho+\lambda)}\right), \quad  i\in\{1,2\}\,,  \lambda\geq 0.
\end{equation}

\begin{proposition}
\label{Proposition 1}
After the competitor's arrival, the  firm's value function defined in \eqref{expression_tildeV} is given by
\begin{equation}
\label{solution tilde V}
\widetilde{V}(x)=
\begin{cases}
-\frac{\alpha}{k_{1}(\rho-\mu)} (\widetilde{a}^{*})^{1-k_{1}}x^{k_{1}} + \frac{\alpha}{\rho-\mu}x-\frac{\beta}{\rho} &\mbox{if $x> \widetilde{a}^{*}$,}\\
0 &\mbox{if $0<x\leq\widetilde{a}^{*}$,}
\end{cases}
\end{equation}
where $k_{1}=h_1(0)<0$ (see \eqref{general}), and the optimal threshold is explicitly given by
\begin{align}
\label{astar}
\widetilde{a}^{*}&=\frac{\rho-\mu}{\rho}\frac{k_{1}}{k_{1}-1}\frac{\beta}{\alpha}>0.
\end{align}
\end{proposition}

\vspace{10pt}

The firm's optimal timing to abandon is given by the first time the value function  $\widetilde{V}(X_t)$ reaches zero.  This occurs when $X_t =  \widetilde{a}^{*}$, so we have \begin{equation}
\label{stopping time tilde a}
\widetilde{\tau}_{a}^{*}=\inf\{t\geq 0: X_{t}\leq \widetilde{a}^{*}\}.
\end{equation}
In particular, the optimal abandonment threshold $\widetilde{a}^{*}$ is slightly smaller than the ratio ${\beta}/{\alpha}$, but note that   $g(x)<0$ for $x<\beta/\alpha$. That means that,  even if the firm  is currently incurring  a loss, it does not immediately abandon but will wait for the profit to improve in the future. When the price falls further to the lower level $\widetilde{a}^{*}$, then the firm will decide to abandon. The  explicit expression of $\widetilde{a}^{*}$ is amenable for sensitivity analysis, which we  summarize as follows.

\begin{corollary}
\label{corollary 1}
The optimal threshold $\widetilde{a}^{*}$ is increasing  in $\rho$, $\beta$, and decreasing  in $\mu$, $\sigma$, $\alpha$.
\end{corollary}

\vspace{10pt}

From Proposition \ref{Proposition 1}, we see  that $\widetilde{V}(x)$ is increasing and  becomes nearly linear when $x$ is large since the term $x^{k_1}$  ($k_1<0$)  diminishes to zero. In fact, it's dominated by the linear part, even if $x$ is close to $\widetilde{a}^{*}$ (see the proof of Corollary \ref{corollary 2} in Appendix \ref{proof}).

\begin{corollary}
\label{corollary 2}
$\widetilde{V}(x)$ is    increasing convex in $x$. In addition, when $x$ is sufficient large, $\widetilde{V}(x)$ is increasing  in $\mu$, $\alpha$, and decreasing in $\rho$, $\beta$.
\end{corollary}
\vspace{10pt}

Next we turn to the optimal abandonment problem   $V(x)$. The explicit solution of $\widetilde{V}(x)$ will become the input to the problem for $V(x)$. We derive the optimal threshold-type strategy. Let   $a^{*}$ denote  the abandonment price level before new  competition arrives. New competition will bring about a reduction in revenue for the firm, described by the fraction $\alpha$. If the revenue impact is low, then  the firm will choose to stay longer in the market and exit in a lower threshold, and thus  we  expect $a^{*}>\widetilde{a}^{*}$. In contrast, when the firm's revenue is significantly reduced, then it is more likely that the firm will experience a negative cash flow, and may be forced to leave the market earlier. This means that the post-competition threshold    $\widetilde{a}^{*}$ is higher than the pre-competition level  $a^{*}$. As is intuitive, the aforementioned  first and  second cases correspond to, respectively,  the large and small  values of   $\alpha$. As we show below, the   two cases are separated by the  critical value  of $\alpha_{0}$, defined by
\begin{equation}
\label{alpha0}
\alpha_{0}=\left(1-\frac{\rho p_{1}(k_1-1)(\rho+\lambda_2-\mu)}{k_{1}(p_{1}-1)(\rho-\mu)(\rho+\lambda_{2})}(1-\frac{K}{\beta})\right)^{-1},
\end{equation}
where $p_1=h_1(\lambda_2)$ as in \eqref{general}.

\begin{proposition}
\label{Proposition 2}
The firm's pre-competition  optimal abandonment  problem \eqref{expression_V} is solved as follows:

(I) If $K \geq \beta$ and $\alpha_{0}\leq\alpha\leq 1$, then there exists  a unique  $a^{*}\in [\widetilde{a}^{*},+\infty)$ such that
\begin{equation}
\label{a1}
\frac{\alpha(k_1-p_1)}{k_1(\rho-\mu)}(\widetilde{a}^{*})^{1-k_{1}}(a^{*})^{k_{1}}+\frac{\rho+\alpha \lambda_{2}-\mu}{\rho+\lambda_{2}-\mu}\frac{p_{1}-1}{\rho-\mu}a^{*}-\frac{p_{1}(\beta \lambda_{2}+\rho K)}{\rho(\rho+\lambda_{2})}=0.
\end{equation}
The value function $V(x)$ is given by
\begin{equation}
\label{solution V 1}
V(x)=
\begin{cases}
C_{1}x^{p_{1}}-\frac{\alpha}{k_{1}(\rho-\mu)}(\widetilde{a}^{*})^{1-k_{1}}x^{k_{1}}+\frac{\rho+\alpha \lambda_{2}-\mu}{\rho+\lambda_{2}-\mu}\frac{1}{\rho-\mu}x-\frac{\beta \lambda_{2}+\rho K}{\rho(\rho+\lambda_{2})} &\mbox{if $x>a^{*}$,}\\
0 &\mbox{if $0<x\leq a^{*}$.}\\
\end{cases}
\end{equation}

(II) If  $K < \beta$, or $K \geq \beta$ and $0\leq\alpha<\alpha_{0}$, then there exists a unique $a^{*}\in (0,\widetilde{a}^{*})$ such that
\begin{equation}
\label{a2}
\frac{(p_{1}-k_{1})\rho(\rho+\lambda_{2}-\mu)+k_{1}\mu\lambda_{2}(1-p_{1})}{(1-k_{1})\rho(\rho+\lambda_{2})(\rho+\lambda_{2}-\mu)}\beta(\widetilde{a}^{*})^{-p_2}(a^{*})^{p_{2}}
+\frac{p_{1}-1}{\rho+\lambda_{2}-\mu}a^{*}-\frac{p_{1}K}{\rho+\lambda_{2}}=0.
\end{equation}
The value function $V(x)$ is given by
\begin{equation}
\label{solution V 2}
V(x)=
\begin{cases}
C_{2}x^{p_{1}}-\frac{\alpha}{k_{1}(\rho-\mu)}(\widetilde{a}^{*})^{1-k_{1}}x^{k_{1}}+\frac{\rho+\alpha \lambda_{2}-\mu}{\rho+\lambda_{2}-\mu}\frac{1}{\rho-\mu}x-\frac{\beta \lambda_{2}+\rho K}{\rho(\rho+\lambda_{2})} &\mbox{if $x>\widetilde{a}^{*}$,}\\
C_{3}x^{p_{1}}+C_{4}x^{p_{2}}+\frac{1}{\rho+\lambda_{2}-\mu}x-\frac{K}{\rho+\lambda_{2}} &\mbox{if $a^{*}<x\leq\widetilde{a}^{*}$,}\\
0 &\mbox{if $0<x\leq a^{*}$.}\\
\end{cases}
\end{equation}
In both cases (I) and (II), the firm's  optimal abandonment time without competition is given by \begin{equation}
\label{stopping time a}
\tau_{a}^{*}=\inf\{t\geq 0: X_{t}\leq a^{*}\}.
\end{equation}

\noindent The coefficients in above expressions are given by
\begin{align}
\label{C1}
C_{1}&=\frac{-v_{1}(a^{*})}{(a^{*})^{p_{1}}}, \\
\label{C3}
C_{2}&=\frac{(p_{1}-k_{1})\rho(\rho+\lambda_{2}-\mu)+k_{1}\mu\lambda_{2}(1-p_{1})}{(k_{1}-1)\rho(\rho+\lambda_{2})(\rho+\lambda_{2}-\mu)}
\frac{((\widetilde{a}^{*})^{p_2-p_1}-(a^{*})^{p_2-p_1})\beta}{(p_{2}-p_{1})(\widetilde{a}^{*})^{p_{1}}} \notag\\
&+\frac{v_{2}(\widetilde{a}^{*})-v_{1}(\widetilde{a}^{*})}{(\widetilde{a}^{*})^{p_{1}}}-\frac{v_{2}(a^{*})}{(a^{*})^{p_{1}}},
\\
\label{C5}
C_{3}&=\frac{1}{1-(a^{*})^{p_1-p_2}(\widetilde{a}^{*})^{p_2-p_1}}\left(C_2+\frac{(\widetilde{a}^{*})^{p_2-p_1}}{(a^{*})^{p_2}}v_2({a^{*})}-\frac{v_{2}(\widetilde{a}^{*})-v_{1}(\widetilde{a}^{*})}{(\widetilde{a}^{*})^{p_{1}}}\right), \\
\label{C6}
C_{4}&=-\frac{C_{3}(a^{*})^{p_1}+v_2({a^{*}})}{(a^{*})^{p_2}},
\\
v_{1}(x)&=-\frac{\alpha}{k_{1}(\rho-\mu)}(\widetilde{a}^{*})^{1-k_{1}}x^{k_{1}}+\frac{\rho+\alpha \lambda_{2}-\mu}{\rho+\lambda_{2}-\mu}\frac{1}{\rho-\mu}x-\frac{\beta \lambda_{2}+\rho K}{\rho(\rho+\lambda_{2})},
\\
v_{2}(x)&=\frac{1}{\rho+\lambda_{2}-\mu}x-\frac{K}{\rho+\lambda_{2}}, \quad \text{ and } \quad  p_2=h_2(\lambda_2).
\end{align}

\end{proposition}

 \vspace{10pt}

Proposition \ref{Proposition 2}  gives  the exact conditions to determine  the ordering of the thresholds $a^*$ and $\widetilde{a}^{*}$.  As an example, let $\beta = \alpha K$ so that  the post-competition profit stream  $g(x)$ be  proportional to $f(x)$, i.e.  $g(x) = \alpha (x -  K)$. Proposition \ref{Proposition 2}  indicates that  $\alpha>\alpha_{0}$, and thus we have $a^{*}>\widetilde{a}^{*}$  in case (II).  In this example,  the firm will have a smaller profit stream after the competitor's arrival. To compensate the reduction in  profit, the firm  opts to stay longer in the market with the competitor and exit   at a lower threshold.

To determine which case applies, the simple first step is to check  whether $K$ is larger than $\beta$ or not. If $K<\beta$, it falls into case (II). Otherwise, we need to  check the values of the parameters   $\alpha$ and $\alpha_{0}$ (see \eqref{alpha0}). Note if $K=\beta$, then $\alpha_0=1$, which  can be seen directly from \eqref{alpha0}.   Next, Corollary \ref{corollary 3} describes the behavior of $\alpha_0$ and its sensitivity in the arrival rate of  competition $\lambda_2$.

\begin{corollary}
\label{corollary 3}
Suppose $K>\beta$, then we have $0<\alpha_0<1$. Moreover, $\alpha_0$ is strictly decreasing with respect to $\lambda_2$, and it admits a limit
\begin{equation}
\label{alpha0_limit}
\alpha^{\infty}_{0}:=\lim_{\lambda_{2}\to+\infty}\alpha_{0}=\left(1+\frac{\rho(1-k_1)}{k_1(\rho-\mu)}(1-\frac{K}{\beta})\right)^{-1} \in (0,1).
\end{equation}
\end{corollary}

\vspace{10pt}

This also leads to the curious question:  under what conditions does  the abandonment level stay unchanged after the competitor's entry? From the proof of Proposition \ref{Proposition 2} in the Appendix, we   see that $a^{*}=\widetilde{a}^{*}$ if and only if $K\geq \beta$ and $\alpha=\alpha_{0}$. As a concrete example, let $\beta=K$ and $\alpha=1$, then $a^{*}=\widetilde{a}^{*}$. This is intuitive since in this case  the competitor's entry does not affect the firm's profit stream, and thus its abandonment timing.

\section{Optimal Entry Timing Problem}\label{sect-entry}
We now analyze the firm's optimal entry timing problem $\psi(x)$ defined in \eqref{expression_psi}. Suggested by  Figure \ref{fig:sample1}, the firm can choose to cancel the project, enter the market, or wait at anytime during the incubation period.  This leads us to determine the  upper and lower  thresholds,  $e^{*}$ and $c^{*}$, for  entry and cancellation by the firm.

\begin{proposition}
\label{Proposition 3} The firm's  optimal entry timing problem \eqref{expression_psi} is given by
\begin{equation}
\label{solution psi}
\psi(x)=
\begin{cases}
V(x) &\mbox{if $x\ge e^{*}$,}\\
J(x) &\mbox{if $c^{*}< x< e^{*}$,}\\
0 &\mbox{if $0<x\le c^{*}$,}\\
\end{cases}
\end{equation}
where
\begin{align}
\label{solution J}
J(x)&=D_{1}x^{q_{1}}+D_{2}x^{q_{2}}-\frac{a}{\rho+\lambda_{1}-\mu}x-\frac{b}{\rho+\lambda_{1}},
\end{align}
and $q_{i}=h_i(\lambda_1)$ as in \eqref{general}.
The constants $\{e^{*}, c^{*}, D_{1}, D_{2}\}$ are found  from the   system of equations:
\begin{align}
\label{ce1}
D_{1}(c^{*})^{q_{1}}+D_{2}(c^{*})^{q_{2}}-\frac{ac^{*}}{\rho+\lambda_{1}-\mu}-\frac{b}{\rho+\lambda_{1}}&=0, \\
\label{ce2}
D_{1}q_{1}(c^{*})^{q_{1}-1}+D_{2}q_{2}(c^{*})^{q_{2}-1}-\frac{a}{\rho+\lambda_{1}-\mu}&=0, \\
\label{ce3}
D_{1}(e^{*})^{q_{1}}+D_{2}(e^{*})^{q_{2}}-\frac{ae^{*}}{\rho+\lambda_{1}-\mu}-\frac{b}{\rho+\lambda_{1}}&=V(e^{*}), \\
\label{ce4}
D_{1}q_1(e^{*})^{q_{1}-1}+D_{2}q_2(e^{*})^{q_{2}-1}-\frac{a}{\rho+\lambda_{1}-\mu}&=\frac{d}{dx}V(x)|_{x=e^{*}}.
\end{align}
The corresponding optimal entry and cancellation times are, respectively,
\begin{align}
\label{stopping time e}
\tau_{e}^{*}=\inf\{t\geq 0: X_{t}\geq e^{*}\}
\qquad \text{ and } \qquad
\tau_{c}^{*} =\inf\{t\geq 0: X_{t}\leq c^{*}\}.
\end{align}
\end{proposition}

\vspace{10pt}
\begin{remark}
To our best knowledge,  the above system of  nonlinear equations do not admit  explicit  solutions. Even in the simpler perpetual American installment call/put option problem studied in \cite{Ciurlia2009,Kimura2009}, similar systems of equations arise and no explicit  solutions are given. Nevertheless, these equations can be solved efficiently  by standard root-finding methods, such as  the Newton-Raphson method, available in many   computational softwares.  \end{remark}

\vspace{10pt}

\begin{figure}[htbp]
\centering \quad
\subfigure[{Case (I): $\widetilde{a}^{*}\leq a^{*}$.}]{\label{fig:1} \includegraphics[width=3in,trim=1cm 0.1cm 0.2cm 0.2cm, clip=true]{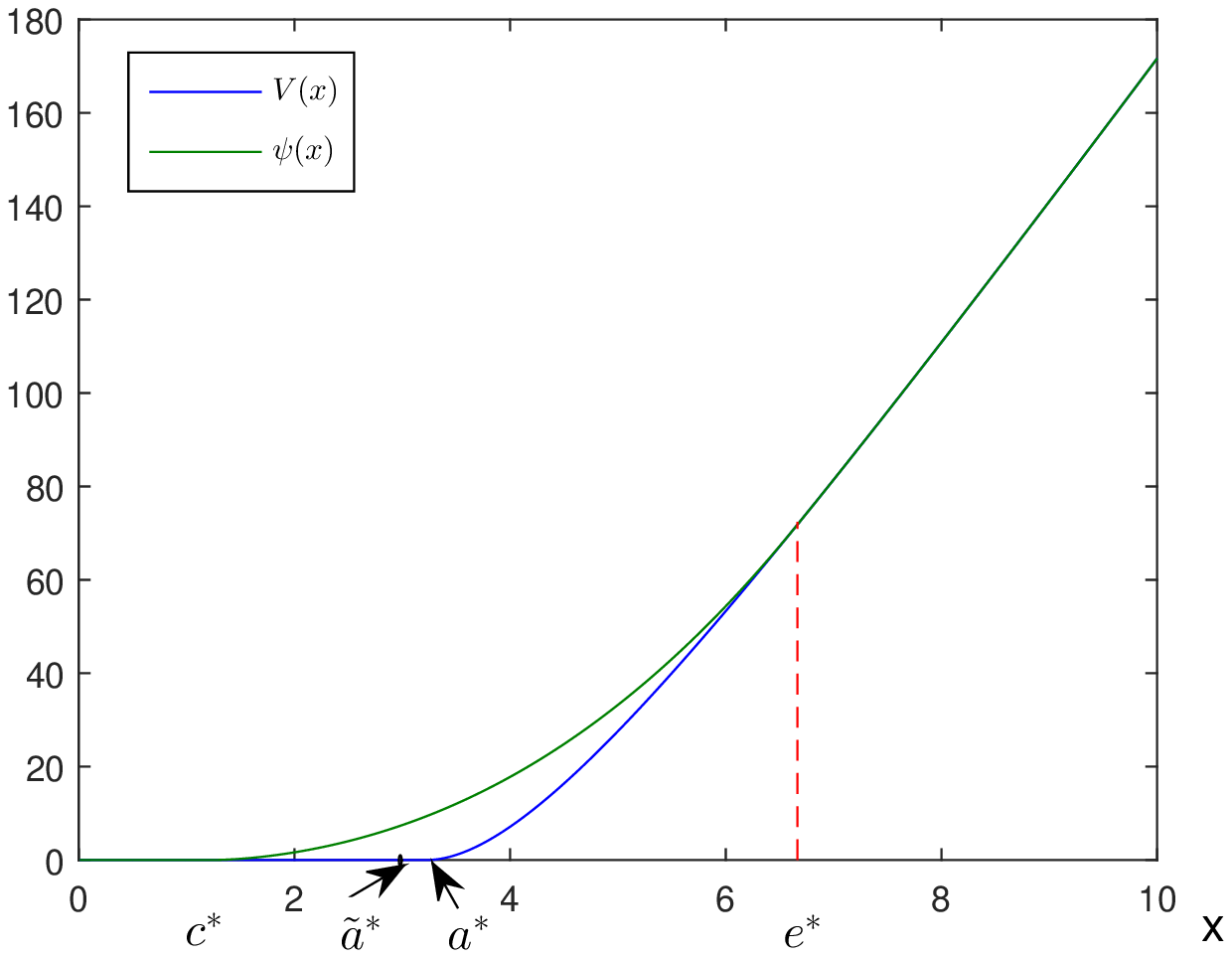} }
\subfigure[Case (II): $\widetilde{a}^{*}>a^{*}$.]{\label{fig:2}
\includegraphics[width=3in,trim=1cm 0.1cm 0.2cm 0.2cm, clip=true]{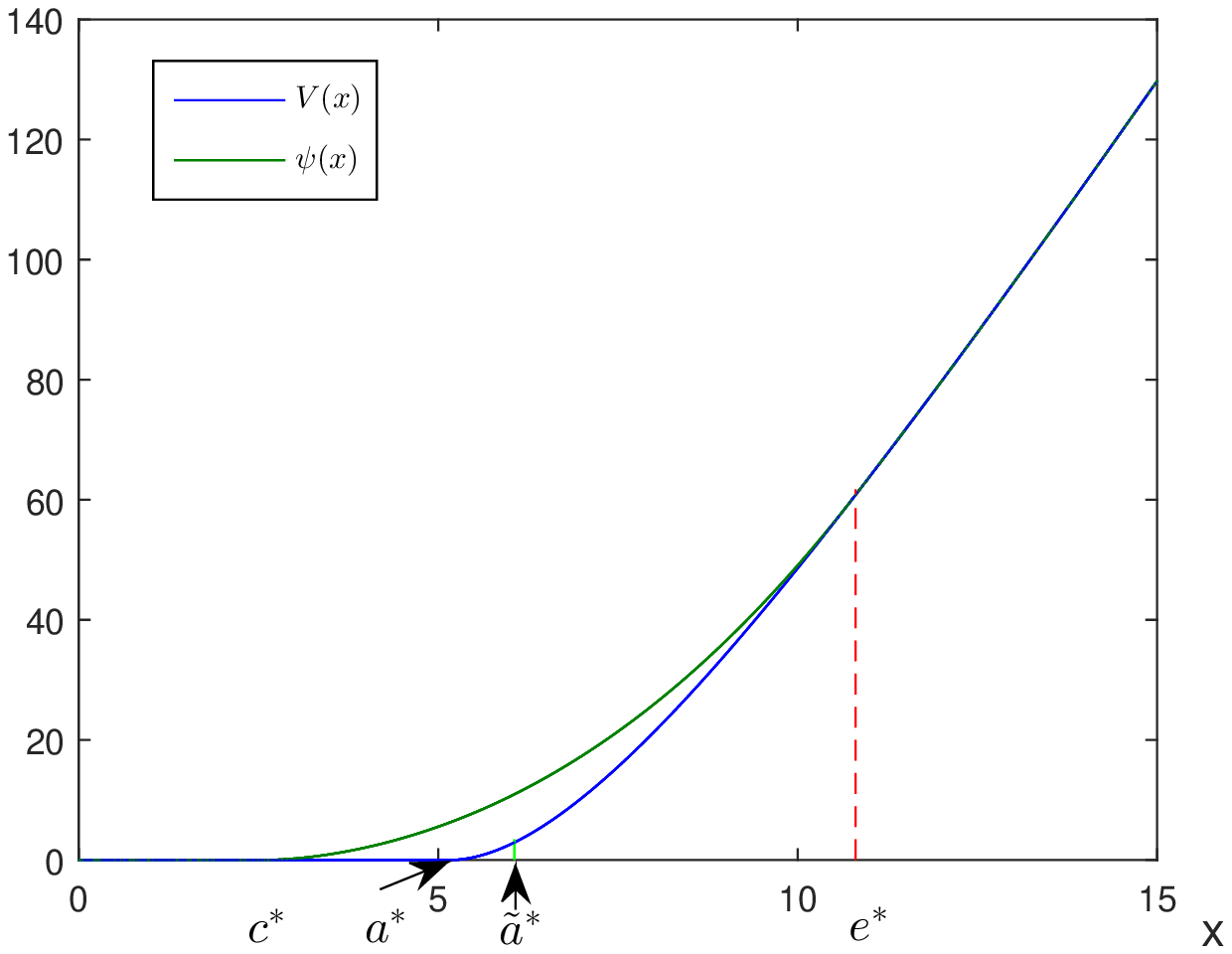} }
\caption{\small{The value functions $\psi(x)$ and $V(x)$ in cases (I) and (II). Left:  $c^{*}=1.21$, $\widetilde{a}^{*}=3.03$, $a^{*}=3.22$ and $e^{*}=6.66$. Right:   $c^{*}=2.55$, $\widetilde{a}^{*}=6.06$, $a^{*}=5.16$ and $e^{*}=10.81$. Common parameters: $\mu=0.03$, $\sigma=0.2$, $\rho=0.05$, $a=0.1$, $b=0.1$, $K=10$, $\lambda_{1}=0.1$, $\lambda_{2}=0.2$,  $\beta=7$, and  $\alpha=0.6$ (left), $0.3$ (right).}}
\label{fig:Plot of Examples}
\end{figure}

The value function $\psi(x)$ depends on the early termination risk parameter $\lambda_1$ as we can see from $J(x)$ in \eqref{solution J} and the coefficients  in \eqref{ce1}-\eqref{ce4}.  We observe from
 \eqref{solution psi} that $\psi(x)$ equals to $V(x)$  for $x \ge e^*$.  Therefore, the competition risk, which exists after the firm's entry,  has an indirect effect on the firm's decisions prior to entering the market since $V(x)$, and thus $e^*$ and $c^*$, depend on $\lambda_2$. Moreover, noticing that $V(x)$ appears in   \eqref{ce1}-\eqref{ce4}, $a^{*}$ enters the system of equations for $e^{*}$ and $c^{*}$ indirectly through the value function $V(x)$ that is tied to $a^{*}$. As a result, both $e^{*}$ and $c^{*}$ depend on $a^{*}$.

Figure \ref{fig:Plot of Examples} shows the value functions  $\psi(x)$ and $V(x)$, along  with four thresholds $c^{*}$, $\widetilde{a}^{*}$, $a^{*}$, and $e^{*}$.   From both panels, we see that $\psi(x)$ and $V(x)$ are both increasing in $x$ as  a higher expected net present value is attained at a higher price level. The value function $\psi(x)$ dominates $V(x)$ due to the timing option (to enter the market) embedded in $\psi(x)$. If the firm chooses to enter immediately at some $x$, then  we have $\psi(x) = V(x)$. This occurs for $x\ge e^{*}$.  By checking the condition in Proposition \ref{Proposition 2}, we have $K=10>\beta=7>0$ and $\alpha_0=0.47$, and in turn,  $\alpha=0.6>0.47=\alpha_0$ in Figure \ref{fig:1} while $\alpha=0.3<0.47=\alpha_0$ in Figure \ref{fig:2}.  In other words,  the left and right panels represent case (I) with $a^{*}\ge \widetilde{a}^{*}$    and case (II)  with $\widetilde{a}^{*}>a^*$  in Proposition \ref{Proposition 2} respectively.

Also, we observe that the cancellation threshold   $c^{*}$ is smaller than both post-entry abandonment thresholds $a^{*}$ and $\widetilde{a}^{*}$. During the incubation period, the option to enter induces the  firm to be willing to incur a loss in order  to wait for the opportunity to enter the market.  However, once the firm has entered the market, the entry option vanishes, and the firm demands more profit to stay in the market, and will exit at a level higher than $c^{*}$.

\vspace{10pt}
\section{Sensitivity Analysis}\label{sect-sen}

In this section, we examine the effects of the early termination risk and the competitor's arrival on the firm's optimal strategies. Common parameters are of the same values as in Figure \ref{fig:Plot of Examples}.  
\subsection{Early Termination}
First,  we conduct a sensitivity analysis of the thresholds with respect to the early termination rate $\lambda_{1}$. As is seen in Figure \ref{fig:sensitivity with respect to lambda1}, the pre/post-competition  abandonment levels do not change with  $\lambda_{1}$ since the firm has already entered the market.   For the entry and cancellation thresholds,   we see that  $e^{*}$  decreases but  $c^{*}$   increases as $\lambda_1$ increases. In other words,  a high early termination risk induces the firm to enter the market early, even if that means capturing a lower profit stream upon entry. Moreover, the firm may also  cancel the project earlier because the expected profitability is reduced by the higher termination risk.  In Figure \ref{fig:sensitivity with respect to lambda1},  we also observe that the threshold $e^{*}$ is higher than  $a^{*}$ while  $c^{*}$ is lower than  $a^{*}$.

Figures \ref{fig:3} and \ref{fig:4} depict the two cases   $\widetilde{a}^{*}\leq a^{*}$ and $\widetilde{a}^{*}> a^{*}$ respectively.  The impact of early termination risk is different  in these two scenarios.  If the competitor's entry causes a significant  shrinkage in the firm's profit  as in case (II), then the post-competition abandonment level  $\widetilde{a}^{*}$ will be higher than  the pre-competition one $a^{*}$. Also, the firm's entry level  $e^{*}$ is decreasing in   $\lambda_{1}$. In Figure \ref{fig:4}, we see that  $e^{*}$ will eventually fall below $\widetilde{a}^{*}$ but  stay above     $a^*$.  In this case, whenever the competitor  enters,   the firm switches its  abandonment level  to the higher level $\widetilde{a}^*$. This forces  the firm to abandon immediately involuntarily if the current value of $X$ is below $\widetilde{a}^*$.

\begin{figure}[h]
	\centering
	\subfigure[Case (I): $\widetilde{a}^{*}\leq a^{*}$. $\alpha=0.6$, $\beta=7$.]{\label{fig:3} \includegraphics[width=3.1in,trim=0.8cm 0.1cm 0.2cm 0.2cm, clip=true]{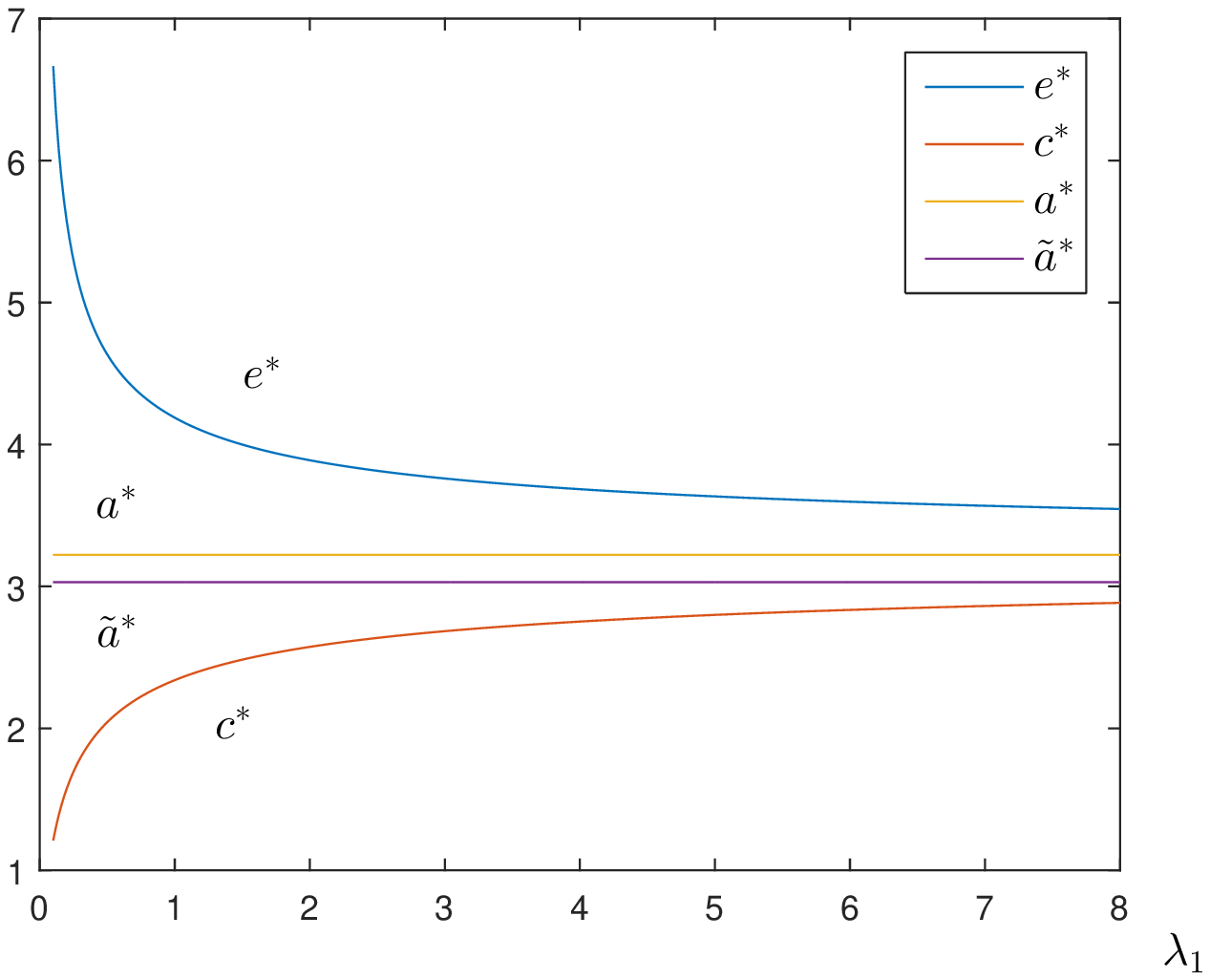} }
	\subfigure[Case (II): $\widetilde{a}^{*}>a^{*}$. $\alpha=0.3$, $\beta=7$.]{\label{fig:4}
		\includegraphics[width=3.1in,trim=0.8cm 0.1cm 0.2cm 0.2cm, clip=true]{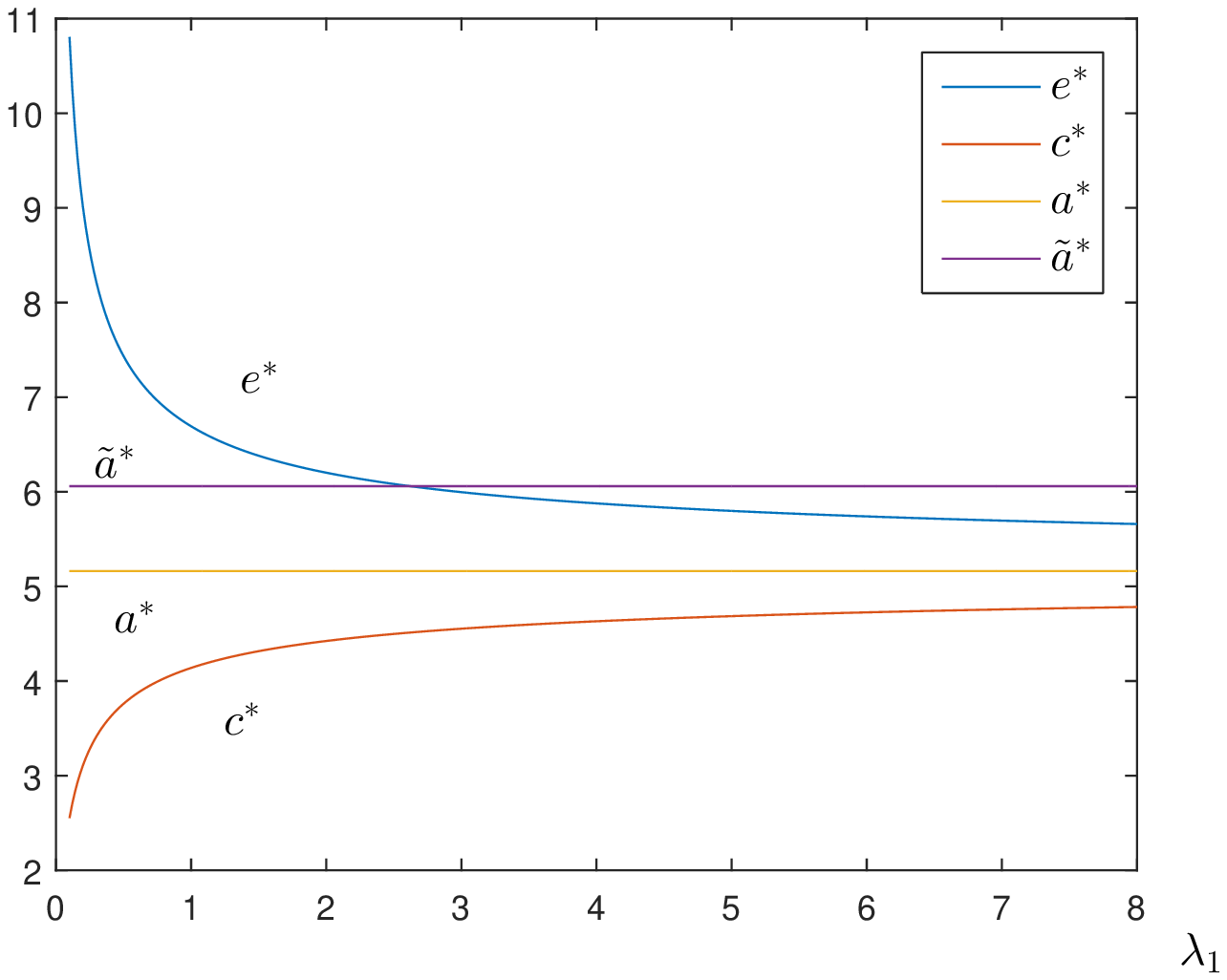} }
	\caption{\small{Sensitivity of the firm's strategies with respect to early termination rate $\lambda_{1}$.}}
	\label{fig:sensitivity with respect to lambda1}
\end{figure}

\subsection{Competition Risk}

The risk of competition  will influence the firm's pre-competition  exit strategy, described by $a^*$, in a non-trivial way. We discuss the four scenarios of different  sensitivities in Figure \ref{fig:sensitivity with respect to lambda2}. In the following discussion, we keep in mind that $\widetilde{a}^*$ doesn't depend on $\lambda_2$ and   $\alpha_{0}$ is decreasing  in $\lambda_{2}$ by  Corollary \ref{corollary 3}.

Figure \ref{fig:6} illustrates  the first scenario with  $K<\beta$. This corresponds to  case (II) in  Proposition \ref{Proposition 2}, where  we have shown that $a^{*}<\widetilde{a}^{*}$ for all $\lambda_{2}>0$.  In addition, $a^{*}$ is increasing  in $\lambda_2$, which means that the firm will choose to  exit early under the threat of  competitor's entry.

For  the remaining  three scenarios in Figures \ref{fig:7}-\ref{fig:9}, we have  $K\ge \beta$. By  Corollary \ref{corollary 3},  we have  $1>\alpha_{0}(0)>\alpha_{0}(\lambda_{2})>\alpha^{\infty}_{0}>0$ for all $\lambda_2>0$. In Figure \ref{fig:7}, $a^*$ is again increasing in $\lambda_2$. In this figure, we have $\alpha\leq\alpha^{\infty}_{0}$, which belongs to case (II) in Proposition \ref{Proposition 2}, and  thus $a^{*}<\widetilde{a}^{*}$.

Figure \ref{fig:8} shows that $a^*$ is first increasing then decreasing in $\lambda_2$. This is because   the constant $\alpha_{0}$ is first larger than $\alpha$ but then smaller than $\alpha$ as $\lambda_{2}$ increases. Therefore, as $\lambda_2$ increases, it changes from case (II) to  case (I)  in  Proposition \ref{Proposition 2}. That explains why $a^{*}$ is smaller than $\widetilde{a}^{*}$ at first, and the relation reverses for large $\lambda_2$.

  Lastly, in Figure \ref{fig:9} we have  $\alpha\geq\alpha_{0}(0)$, which corresponds to  case (I) in Proposition \ref{Proposition 2}.  We have shown that  $a^{*}\geq\widetilde{a}^{*}$ for all $\lambda_{2}$.  In contrast to the other scenarios,  the pre-competition threshold $a^{*}$ is decreasing and  approaches  $\widetilde{a}^{*}$ as $\lambda_2$ increases. In this scenario, the impact of the competitor's entry on the firm's profit is much less than in other scenarios, as indicated by the higher  value of $\alpha$.

In summary, scenarios (a), (b), and the first part of scenario (c) belong to  case (II), in which the  competitor's potential entry will lead to a significant profit reduction. Therefore,  the firm exits early by selecting a higher pre-competition abandonment threshold $a^{*}$ as the competitor's arrival rate increases. The opposite happens to scenario (d) and the second part (large $\lambda_2$) of scenario (c).  As $\lambda_2$ increases,  which means high risk of facing  competition in the future, the firm tries to  stay longer before the competitor's entry to make more profits in case of the potential significant loss afterwards. This results in a decreasing  pre-competition abandonment threshold  $a^{*}$.

Despite different behaviors of $a^{*}$, there is a common phenomenon that $a^{*}$  approaches  $\widetilde{a}^{*}$ as $\lambda_2$ increases. The intuitive explanation is as follows. When the competition risk is very high, the firm  will adjust its abandonment level  $a^{*}$ in order not to be affected too much when the competitor  eventually enters. That means that the firm's  abandonment level $a^{*}$ before the competitor's entry will be  close to the post-competition  abandonment level $\widetilde{a}^{*}$.

\begin{figure}[h]
\centering
\subfigure[$\alpha=0.20$, $\beta=14$.]{\label{fig:6} \includegraphics[width=3.1in,trim=0.6cm 0.1cm 0.2cm 0.2cm, clip=true]{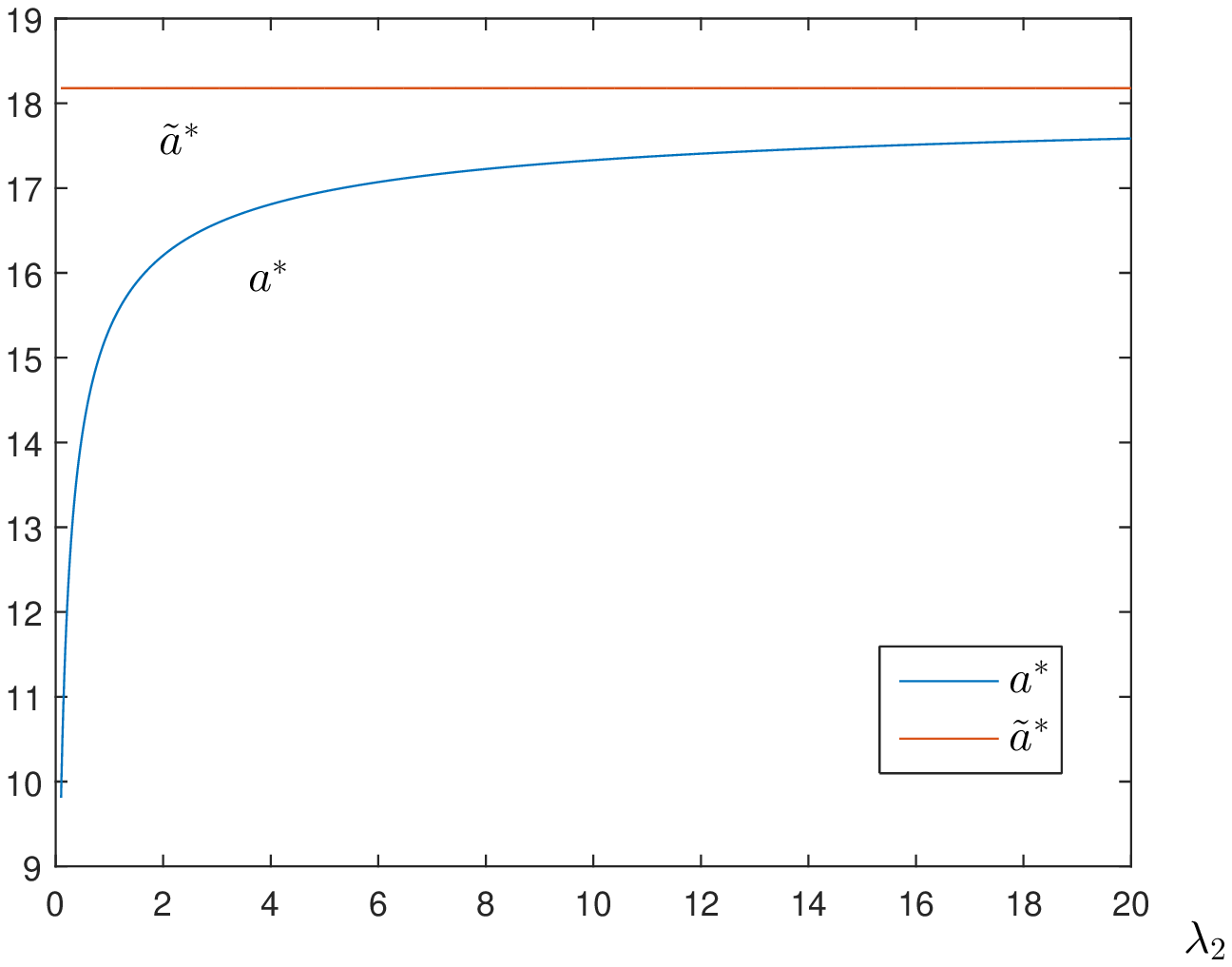} }
\subfigure[$\alpha=0.20$, $\beta=7$.]{\label{fig:7} \includegraphics[width=3.1in,trim=0.6cm 0.1cm 0.2cm 0.2cm, clip=true]{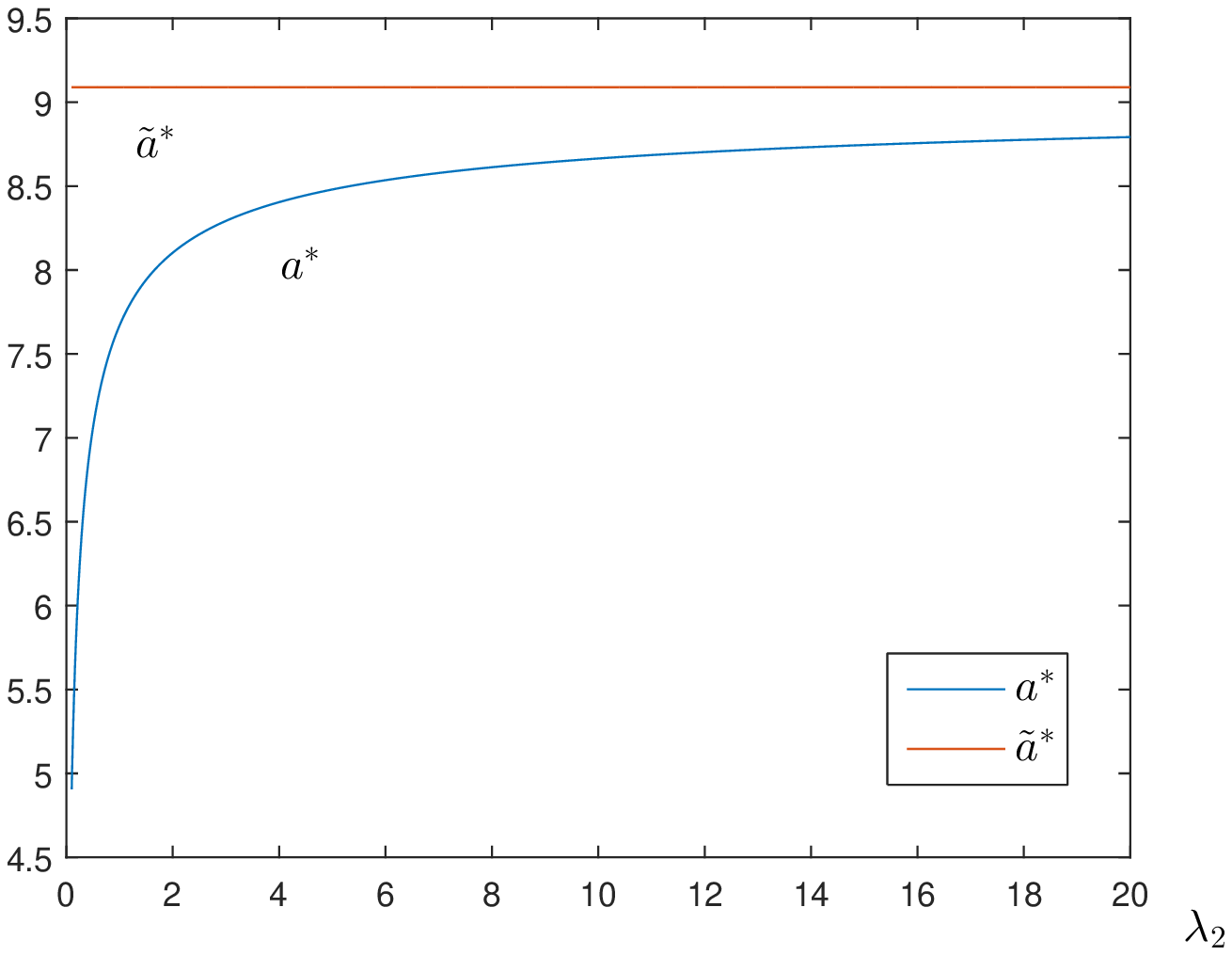} }
\subfigure[$\alpha=0.45$, $\beta=7$.]{\label{fig:8} \includegraphics[width=3.1in,trim=0.6cm 0.1cm 0.2cm 0.2cm, clip=true]{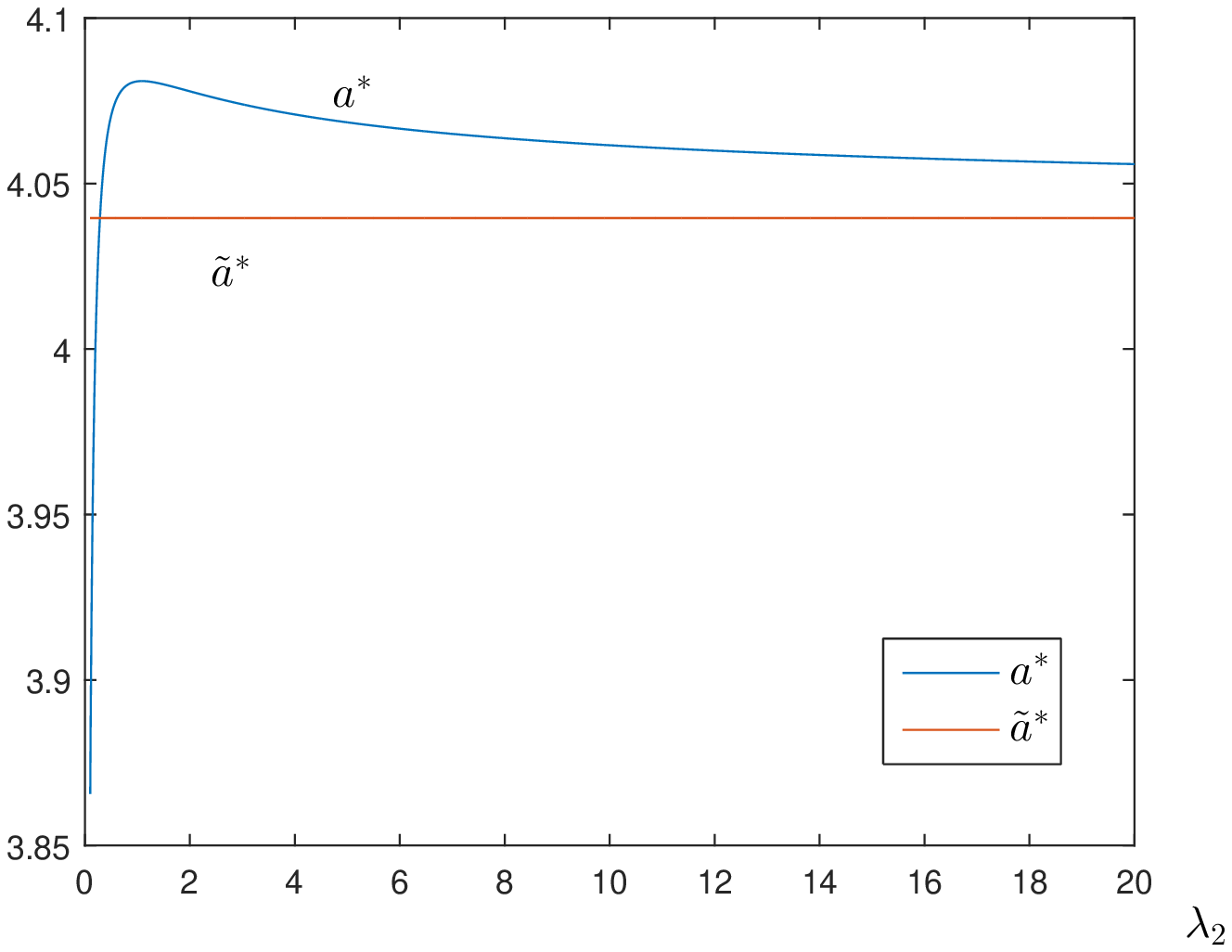} }
\subfigure[$\alpha=0.80$, $\beta=7$.]{\label{fig:9} \includegraphics[width=3.1in,trim=0.6cm 0.1cm 0.2cm 0.2cm, clip=true]{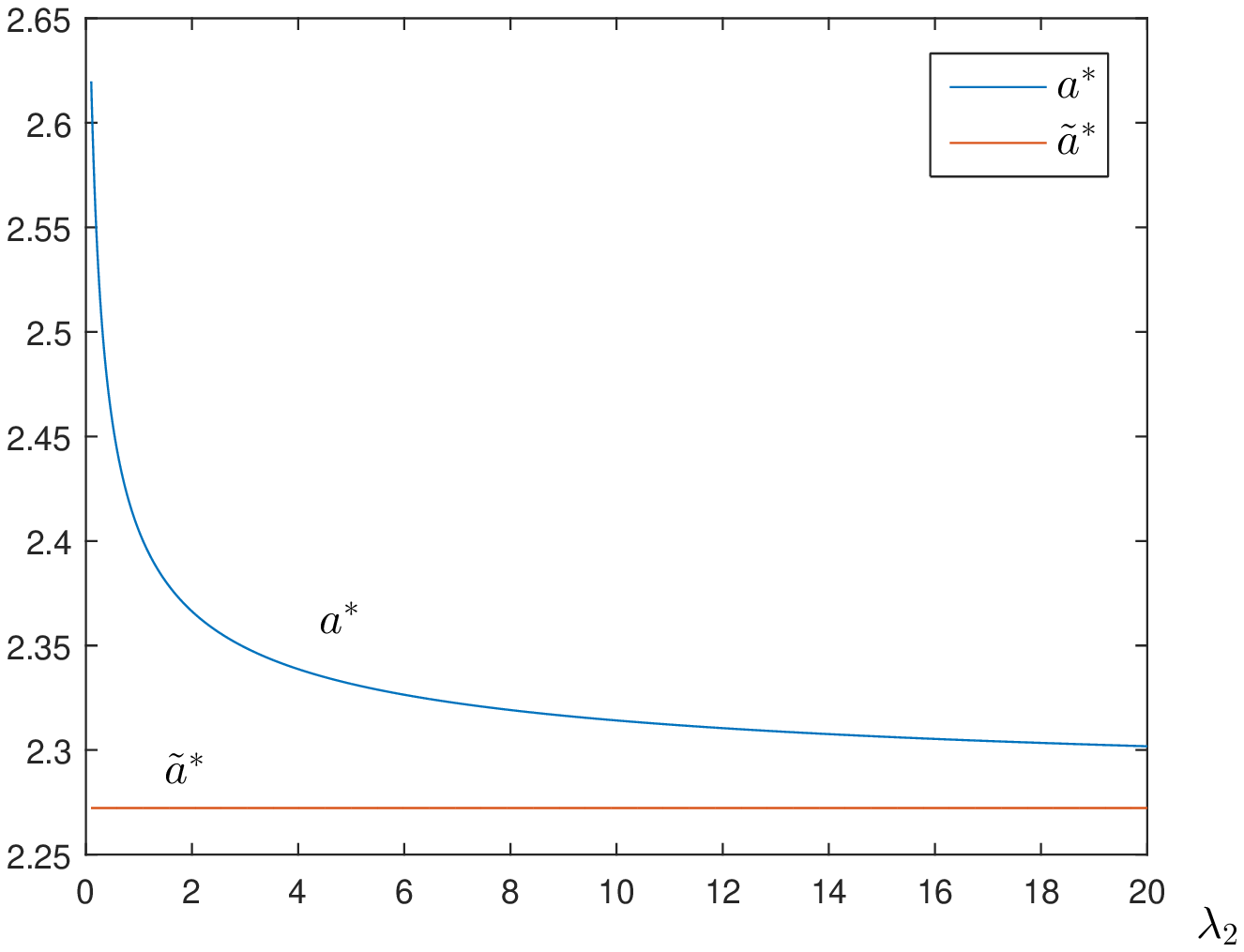} }
\caption{\small{Sensitivity of the  abandonment thresholds, $a^*$ (pre-competition) and $\widetilde{a}^*$ (post-competition), with respect to the competitor's arrival rate $\lambda_{2}$.}}
\label{fig:sensitivity with respect to lambda2}
\end{figure}

In addition to $a^{*}$, the entry and cancellation thresholds $e^{*}$ and $c^{*}$ are also affected by the  competitor's arrival rate  $\lambda_{2}$ after the firm's market entry. In  two scenarios, Figure \ref{fig:10} and Figure \ref{fig:11} respectively take the same parameters as Figure \ref{fig:7} and Figure \ref{fig:9}, so   they represent respectively case (II) and (I) in Proposition \ref{Proposition 2}.  Figure \ref{fig:10} shows that both $e^{*}$ and $c^{*}$  are increasing in $\lambda_{2}$ under case (II). In particular, the entry threshold $e^{*}$ is seen to be lower  than the abandonment threshold $\widetilde{a}^{*}$ when $\lambda_2$ is small, but then surpasses  it  for large  $\lambda_{2}$. In this case, the competitor's arrival will lead to a significant post-entry profit reduction. Therefore, as is intuitive,   the  firm requires  a higher  entry level to ensure a higher  post-entry  profit so as to make the entry worthwhile. Nevertheless, under case (I), Figure \ref{fig:11} shows that both $e^{*}$ and $c^{*}$ are  decreasing in $\lambda_{2}$. In other words, when the post-entry competition arrival is not expected to have a significant impact on profit, the firm will opt to enter the market earlier at a lower entry threshold.

\begin{figure}[h]
	\centering
	\subfigure[$\alpha=0.20$, $\beta=7$.]{\label{fig:10} \includegraphics[width=3.1in,trim=0.6cm 0.1cm 0.2cm 0.2cm, clip=true]{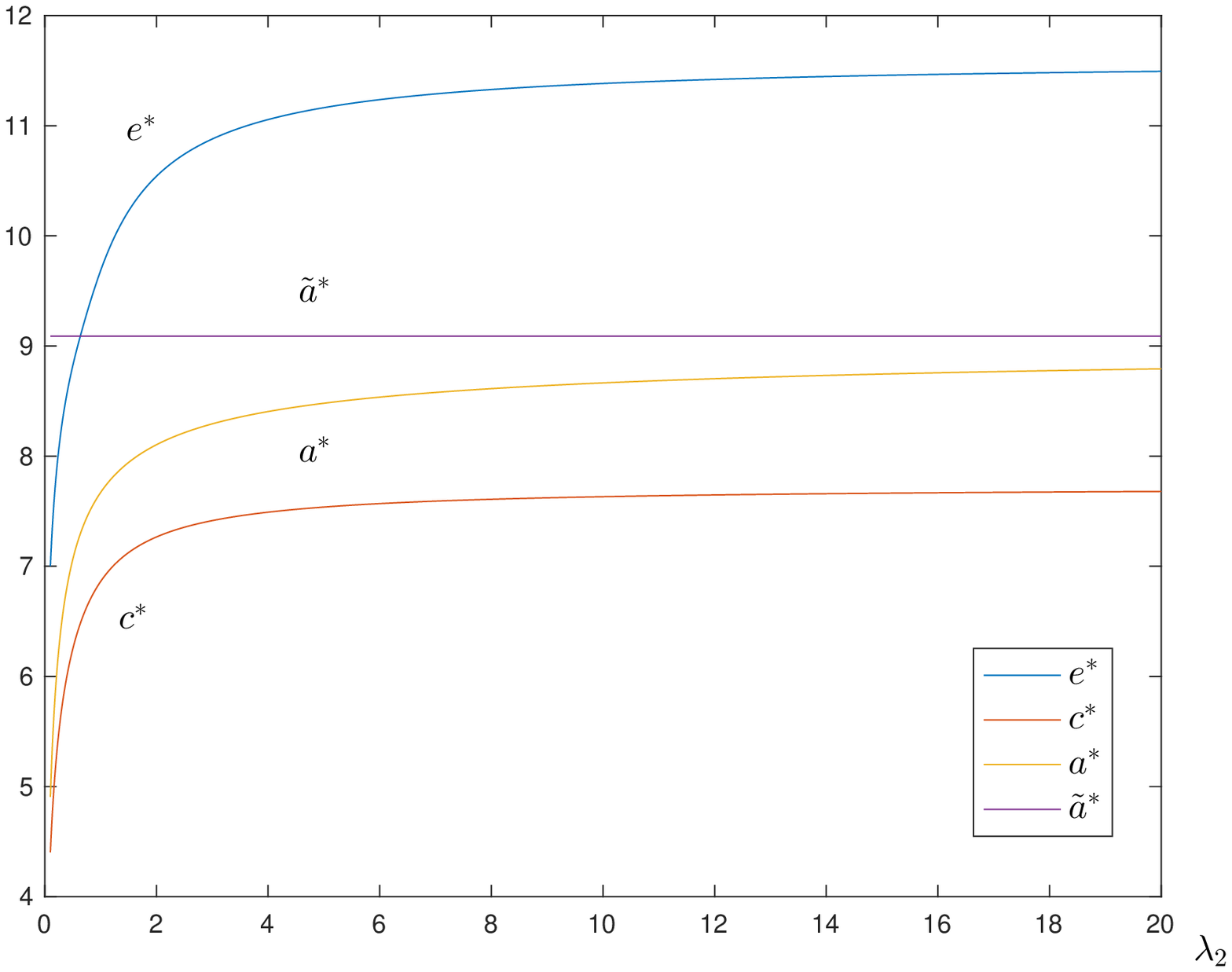} }
	\subfigure[$\alpha=0.80$, $\beta=7$.]{\label{fig:11} \includegraphics[width=3.1in,trim=0.6cm 0.1cm 0.2cm 0.2cm, clip=true]{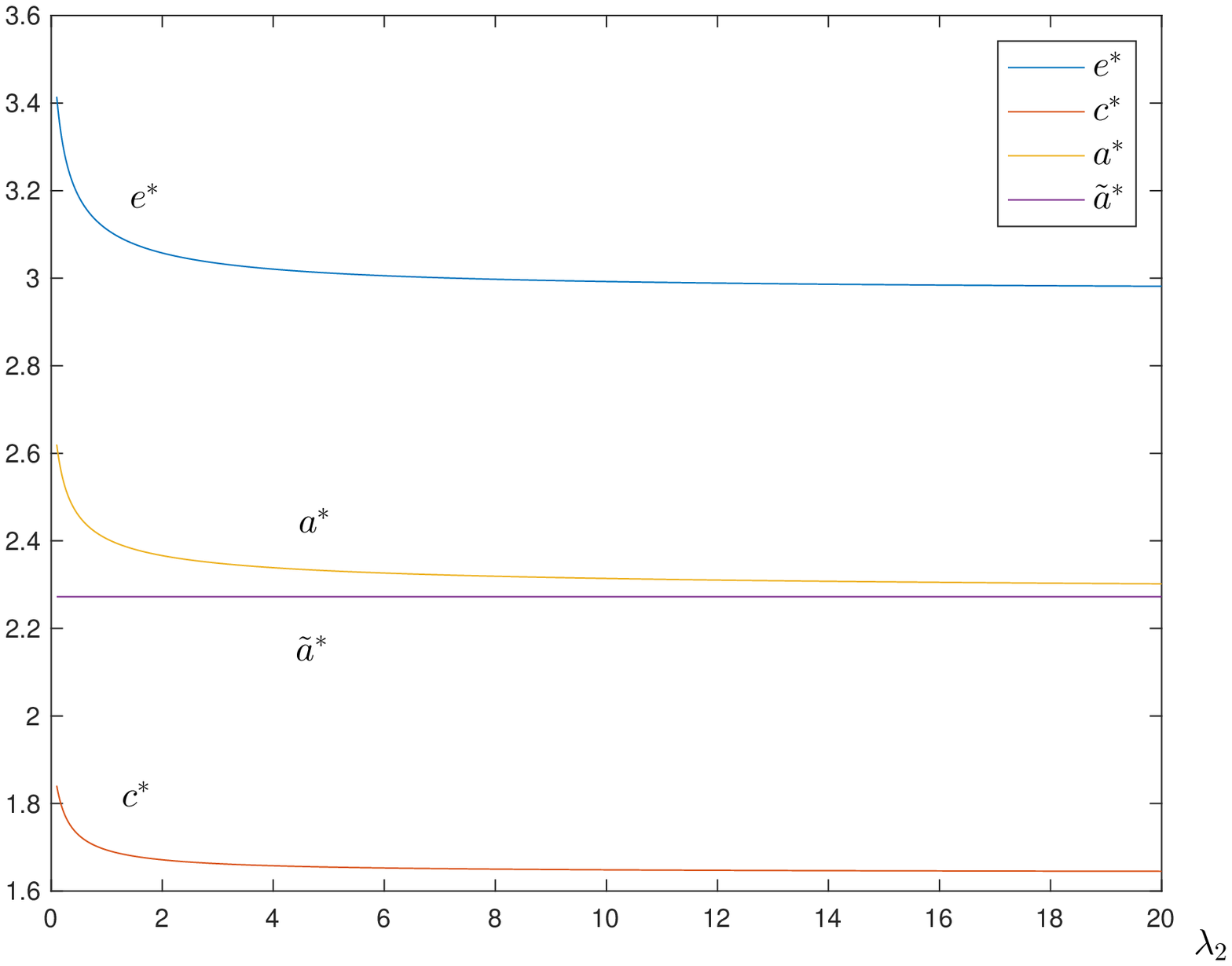} }
	\caption{\small{Sensitivity of the entry and cancellation thresholds, $e^*$ and ${c}^*$, with respect to the competitor's arrival rate $\lambda_{2}$.}}
	\label{fig:e c sensitivity with respect to lambda2}
\end{figure}

\section{Concluding Remarks}\label{sect-conclude}
We have developed a  model to evaluate the timing options  to enter  or  exit a given market for a startup firm. Our model accounts for the risks of early termination prior to market entry, as well as the risk of new competition that may reduce the firm's future profit stream. Our analytic solutions  allow for instant computation of the firm's optimal expected NPV and the associated optimal timing strategies. We have provided explanation on the non-trivial combined effect of early termination and competition risks on the firm's entry and pre-competition abandonment decisions.

 There are a number of directions for future research.  For tractability, we have chosen to work with a lognormal model. Alternatively, we will consider in future work   other dynamics for the underlying price process, such as jump diffusions or mean-reverting processes, such as the Cox-Ingersoll-Ross, and exponential Ornstein-Uhlenbeck processes (\cite{LeungLiWang2014CIR}, \cite{LeungLiWang2015XOU}). For a startup, it is also useful to choose between equity and debt financing and examine the implications to the firm's investment and bankruptcy decisions (see, e.g. \cite{TianYuan}).   Another natural extension of our model is to incorporate  sequential random arrivals and departures of competitors.

\newpage
\appendix
\section{Appendix}\label{proof}
In this appendix, we provide the proofs of all the propositions and corollaries presented above.

\subsection{Proof of Equations (\ref{expression_V}) and (\ref{expression_psi})}\label{app1}
Recall  that  $({\cal{F}}_{t})_{t\geq 0}$ is the filtration  generated by $X$. We re-state equation (\ref{expression_V}), and apply the tower property to get\begin{align}
V(x) &= \sup_{\substack{\tau_{a}\in\cal{T}}}\mathbb{E}_{x}\left[\int_{0}^{\tau_{a}\wedge\zeta_{2}}\ee^{-\rho t}f\left(X_{t}\right)\,dt + \ind_{\{\tau_{a}>\zeta_{2}\}}\ee^{-\rho\zeta_{2}}\widetilde{V}(X_{\zeta_{2}}) \,\right] \nonumber  \\
&= \sup_{\substack{\tau_{a}\in\cal{T}}}\mathbb{E}_{x}\left[\mathbb{E}\left[\int_{0}^{\tau_{a}\wedge\zeta_{2}}\ee^{-\rho t}f\left(X_{t}\right)\,dt + \ind_{\{\tau_{a}>\zeta_{2}\}}\ee^{-\rho\zeta_{2}}\widetilde{V}(X_{\zeta_{2}})~ \Big | {\cal{F}}_{\tau_{a}}\right]\right].\label{Vinapp}
\end{align}
 
The arrival time $\zeta_{2}$ is exponentially distributed and   independent of $X$, so the first term inside the expectation \eqref{Vinapp}  can be written as
\begin{align}
\mathbb{E}\left[\int_{0}^{\tau_{a}\wedge\zeta_{2}}\ee^{-\rho t}f\left(X_{t}\right)\,dt~ \Big | {\cal{F}}_{\tau_{a}}\right] 
& = \mathbb{E}\left[\int_{0}^{\tau_{a}}\ind_{\{\zeta_{2}\geq t\}}\ee^{-\rho t}f\left(X_{t}\right)\,dt~ \Big | {\cal{F}}_{\tau_{a}}\right]\nonumber \\
& = \mathbb{E}\left[\int_{0}^{\tau_{a}}\left(\int_{t}^{+\infty}\lambda_{2}\ee^{-\lambda_{2}y}\,dy\right)\ee^{-\rho t}f\left(X_{t}\right)\,dt~ \Big | {\cal{F}}_{\tau_{a}}\right]  \nonumber \\
& =\mathbb{E}\left[ \int_{0}^{\tau_{a}}\ee^{-(\lambda_{2}+\rho )t}f\left(X_{t}\right)\,dt~ \Big | {\cal{F}}_{\tau_{a}}\right].  \label{first item in V}
\end{align}

By the same procedure, we express  the second term in \eqref{Vinapp} as
\begin{equation}
\mathbb{E}\left[ \ind_{\{\tau_{a}>\zeta_{2}\}}\ee^{-\rho\zeta_{2}}\widetilde{V}(X_{\zeta_{2}}) ~ \Big | {\cal{F}}_{\tau_{a}}\right]  =\mathbb{E}\left[  \int_{0}^{\tau_{a}}\ee^{-(\lambda_{2}+\rho) t}\lambda_{2}\widetilde{V}(X_{t}) \, dt~ \Big | {\cal{F}}_{\tau_{a}}\right] . \label{second item in V}
\end{equation}
Combining (\ref{first item in V}) and (\ref{second item in V}),   we obtain equation (\ref{expression_V}). Equation (\ref{expression_psi}) is derived in the same fashion.

\subsection{Proof of Proposition \ref{Proposition 1}}
\begin{proof}
 We consider a candidate stopping time $\widetilde{\tau}_{a}=\inf\{t\geq 0: X_{t}\leq \widetilde{a}^{*}\}$, where $\widetilde{a}^{*}>0$. In the stopping region $\{x<\widetilde{a}^{*}\}$, $\widetilde{V}(x) = 0$ from the definition of $\widetilde{\tau}_{a}$ and \eqref{expression_tildeV}. In the continuation region $\{x\geq \widetilde{a}^{*}\}$, $\widetilde{V}(x)$ satisfies the following ODE:
\begin{equation}
\label{ODE tilde V}
-\rho\widetilde{V}(x)+\mu x\widetilde{V}'(x)+\frac{\sigma^{2}x^{2}}{2}\widetilde{V}''(x)+g(x)=0,
\end{equation}
with boundary conditions:
\begin{align}
\label{boundary condition tilde V}
\widetilde{V}(\widetilde{a}^{*}) = 0, \,\,\,\,\,\,
\widetilde{V}'(\widetilde{a}^{*}) = 0, \,\,\,\,\,\,
\lim_{x\to\infty}\widetilde{V}'(x) < \infty.
\end{align}

The ODE and the last condition in \eqref{boundary condition tilde V} indicate that the   solution is given by $B_{1}x^{k_{1}}+g_{0}(x)$, where with constants $B_{1}$ and $k_1$, such that     $k_{1}$ is the negative  root of the  associated  quadratic equation, and $g_{0}(x)$ is a   particular solution of \eqref{ODE tilde V}:
\begin{equation}
g_{0}(x)=\frac{\alpha x}{\rho-\mu}-\frac{\beta}{\rho}.
\end{equation}

To determine $\widetilde{a}^{*}$ and $B_{1}$, we utilize boundary conditions \eqref{boundary condition tilde V} to get the   system of equations:
\begin{align}
B_{1}(\widetilde{a}^{*})^{k_{1}}+\frac{\alpha \widetilde{a}^{*}}{\rho-\mu}+\frac{\beta}{\rho}&=0, \\
B_{1}k_{1}(\widetilde{a}^{*})^{k_{1}-1}+\frac{\alpha}{\rho-\mu}&=0.
\end{align}
The above equations   yield \eqref{solution tilde V} and \eqref{astar} in Proposition \ref{Proposition 1}.\end{proof}

\subsection{Proof of Corollary \ref{corollary 1}}

\begin{proof}
First, we define   the ratio:
\begin{align}
\label{second_k1}
\frac{k_1}{k_1-1}&=\left(1+\sigma^2\left(\sqrt[]{(\frac{\sigma^{2}}{2}-\mu)^{2}+2\sigma^{2}\rho}-\frac{\sigma^2}{2}+\mu\right)^{-1}\right)^{-1}.
\end{align}
From \eqref{second_k1} and the expression of $\widetilde{a}^{*}$ in  \eqref{astar}, we know that $\frac{k_1}{k_1-1}$ increases in $\rho$ and thus $\widetilde{a}^{*}$ increases in $\rho$.

Next, we differentiate $\widetilde{a}^{*}$ with respect to $\mu$. Let $\delta:=\sqrt[]{(\frac{\sigma^{2}}{2}-\mu)^{2}+2\sigma^{2}\rho}>|\frac{\sigma^{2}}{2}-\mu|\geq 0$, then
\begin{equation}
\frac{\partial{\widetilde{a}^{*}}}{\partial{\mu}}=\frac{-2\beta}{\alpha(2\rho+\frac{\sigma^{2}}{2}-\mu+\delta)^2}\left(\frac{\sigma^2}{2}+\delta+(\rho\delta+(\rho-\mu)(\mu-\frac{\sigma^2}{2}))
\delta^{-1}\right).
\end{equation}
Since $\rho\delta>|(\rho-\mu)(\mu-\frac{\sigma^2}{2})|$, we have $\frac{\partial{\widetilde{a}^{*}}}{\partial{\mu}}<0$ and thus $\widetilde{a}^{*}$ decreases with respect to $\mu$.

In addition,
\begin{equation}
\frac{\partial k_1}{\partial\sigma}=\frac{\partial k_1}{\partial{(\sigma^2)}}\frac{\partial{(\sigma^2)}}{\partial\sigma}=\frac{\frac{\sigma^2}{2}-\mu+\delta+2\rho}{2\rho}k_1^2\sigma>0.
\end{equation}
Therefore $k_1<0$ and $\frac{\partial k_1}{\partial\sigma}>0$ indicate that $\frac{k_1}{k_1-1}$ decreases in $\sigma$ and thus $\widetilde{a}^{*}$ decreases in $\sigma$.

Finally, that $\widetilde{a}^{*}$ decreases in $\alpha$ and increases in $\beta$ can be directly inferred from \eqref{astar}. \end{proof}

\subsection{Proof of Corollary \ref{corollary 2}}
\begin{proof}
We take the first and second derivatives of $\widetilde{V}(x)$, namely,
\begin{align}
\widetilde{V}'(x)=\frac{\alpha}{\rho-\mu}\left(1-\left(\frac{\widetilde{a}^{*}}{x}\right)^{1-k_1}\right)  \quad \text{and } \quad  \widetilde{V}''(x)=\frac{\alpha(1-k_1)}{\rho-\mu}(\widetilde{a}^{*})^{1-k_1}x^{k_1-2}\,, \quad \forall x>\widetilde{a}^{*}.
\end{align} In turn, $k_1<0$ implies $\widetilde{V}'(x)>0$ and $\widetilde{V}''(x)>0$ for all $x>\widetilde{a}^{*}$. In addition, when $x$ is sufficient large, the first part of $\widetilde{V}(x)$ is negligible and its sensitivity with respect to the  parameters can be deduced   from \eqref{solution tilde V}.
\end{proof}

\subsection{Proof of Proposition \ref{Proposition 2}}

To prove Proposition \ref{Proposition 2}, we first establish the following lemmas.
\begin{lemma}
\label{lemma 1}
Let $H:\mathbb{R}_+\to\mathbb{R}$ such that
\begin{equation}
H(x)=\frac{\alpha(k_1-p_1)}{k_1(\rho-\mu)}(\widetilde{a}^{*})^{1-k_{1}}x^{k_{1}}+\frac{\rho+\alpha \lambda_{2}-\mu}{\rho+\lambda_{2}-\mu}\frac{p_{1}-1}{\rho-\mu}x-\frac{p_{1}(\beta \lambda_{2}+\rho K)}{\rho(\rho+\lambda_{2})},
\end{equation}
then there exists unique $a^{*}\in[\widetilde{a}^{*},+\infty)$, such that $H(a^{*})=0$, if and only if $K \geq \beta$ and $\alpha_{0}\leq \alpha\leq 1$. If $\alpha=\alpha_{0}$, then $a^{*}=\widetilde{a}^{*}$.
\end{lemma}

\begin{proof}
We take the first and second derivatives of $H(x)$,

\begin{align}
H'(x)&=\frac{\alpha(k_1-p_1)}{\rho-\mu}(\widetilde{a}^{*})^{1-k_{1}}x^{k_{1}-1}+\frac{\rho+\alpha \lambda_{2}-\mu}{\rho+\lambda_{2}-\mu}\frac{p_{1}-1}{\rho-\mu}, \\
H''(x)&=\frac{\alpha(k_1-p_1)(k_{1}-1)}{\rho-\mu}(\widetilde{a}^{*})^{1-k_{1}}x^{k_{1}-2}.
\end{align}

Notice $\alpha>0$ and $p_{1}<k_{1}<0$, then $H''(x)<0$ for all $x>0$. This means $H'(x)$ is a decreasing function with respect to $x$. By inspection, we have
\begin{equation}
H'(x)
\begin{cases}
>0 &\mbox{if $0<x<{a}^{*}_{0}$,}\\
=0 &\mbox{if $x={a}^{*}_{0}$,}\\
<0 &\mbox{if $x>{a}^{*}_{0}$,}\\
\end{cases} \quad \text{ with} ~~ {a}^{*}_{0}=\widetilde{a}^{*}\left(\frac{1-p_{1}}{\alpha(k_{1}-p_{1})}\frac{\rho+\alpha\lambda_{2}-\mu}{\rho+\lambda_{2}-\mu}\right)^{\frac{1}{k_{1}-1}}\,.
\end{equation}
The   derivative means that $H(x)$ is unimodal and maximized at ${a}^{*}_{0}$. Moreover, we have $\lim_{x\to 0} H(x)=\lim_{x\to +\infty}H(x)=-\infty$. Note that   $\left(\frac{1-p_{1}}{\alpha(k_{1}-p_{1})}\frac{\rho+\alpha\lambda_{2}-\mu}{\rho+\lambda_{2}-\mu}\right)$ is a decreasing function of $\alpha~(0<\alpha\leq1)$. When $\alpha=1$, $\left(\frac{1-p_{1}}{\alpha(k_{1}-p_{1})}\frac{\rho+\alpha\lambda_{2}-\mu}{\rho+\lambda_{2}-\mu}\right)>1$. Therefore, we see that  ${a}^{*}_{0}<\widetilde{a}^{*}$ for all $0<\alpha\leq 1$.

There exists a unique $a^{*}$ such that $H(a^{*})=0$ and $a^{*}\geq\widetilde{a}^{*}$ if and only if $H(\widetilde{a}^{*})\geq0$. To see this, if there is an $a^{*}\geq\widetilde{a}^{*}$ such that $H(a^{*})=0$, then $H(\widetilde{a}^{*})\geq0$, because $H(x)$ is a decreasing function when $x>{a}^{*}_{0}$ and $a^{*}\geq\widetilde{a}^{*}>{a}^{*}_{0}$. On the other hand, if $H(\widetilde{a}^{*})\geq0$, the existence and uniqueness of $a^{*}$ comes from the fact that $H(x)$ is continuous and decreasing for  $x>\widetilde{a}^{*}>{a}^{*}_{0}$, along with  $\lim_{x\to +\infty}H(x)=-\infty$.

Now $H(\widetilde{a}^{*})$ is given by
\begin{equation}
H(\widetilde{a}^{*})=\frac{\beta(p_{1}-k_{1})}{\rho(1-k_{1})}+\frac{k_{1}\beta(\rho+\alpha \lambda_{2}-\mu)(1-p_{1})}{\rho\alpha(\rho+\lambda_{2}-\mu)(1-k_{1})}-\frac{p_{1}(\beta \lambda_{2}+\rho K)}{\rho(\rho+\lambda_{2})}.
\end{equation}

It's easy to see that  $H(\widetilde{a}^{*})$ is increasing  in $\alpha~(0\leq\alpha\leq1)$, holding other parameters fixed. When $\alpha=1$, $H(\widetilde{a}^{*})|_{\alpha=1}=\frac{p_{1}(\beta-K)}{(\lambda_{2}+\rho)}$, which means $H(\widetilde{a}^{*})\leq\frac{p_{1}(\beta-K)}{(\lambda_{2}+\rho)}$ for all $0\leq\alpha\leq 1$.\\
$(1)$ If $K<\beta$, then $H(\widetilde{a}^{*})<0$, and there is not an $a^{*}\geq\widetilde{a}^{*}$ such that $H(a^{*})=0$.\\
$(2)$ If $K \geq \beta$, then
\begin{equation}
H(\widetilde{a}^{*})
\begin{cases}
>0 &\mbox{if $\alpha>\alpha_{0}$,}\\
=0 &\mbox{if $\alpha=\alpha_{0}$,}\\
<0 &\mbox{if $\alpha<\alpha_{0}$,}\\
\end{cases}\quad \text{ with } ~~ \alpha_{0}=\left(1-\frac{\rho p_{1}(k_1-1)(\rho+\lambda_2-\mu)}{k_{1}(p_{1}-1)(\rho-\mu)(\rho+\lambda_{2})}(1-\frac{K}{\beta})\right)^{-1}\,.
\end{equation}

From cases  $(1)$ and $(2)$ above, we conclude that  $K \geq \beta$ and $\alpha_{0}\leq \alpha\leq 1$ iff $H(\widetilde{a}^{*})\geq0$ iff there exists a unique $a^{*}$ such that $H(a^{*})=0$ and $a^{*}\geq\widetilde{a}^{*}$. Moreover, $a^{*}=\widetilde{a}^{*}$ iff  $K\geq \beta$ and $\alpha=\alpha_{0}$, due to  the uniqueness of $a^{*}$.
\end{proof}


\begin{lemma}
\label{lemma 2}
Let $M:\mathbb{R}_+\to\mathbb{R}$ such that
\begin{equation}
M(x)=\frac{(p_{1}-k_{1})\rho(\rho+\lambda_{2}-\mu)+k_{1}\mu\lambda_{2}(1-p_{1})}{(1-k_{1})\rho(\rho+\lambda_{2})(\rho+\lambda_{2}-\mu)}\beta
(\widetilde{a}^{*})^{-p_2}x^{p_2}+\frac{p_{1}-1}{\rho+\lambda_{2}-\mu}x-\frac{p_{1}K}{\rho+\lambda_{2}},
\end{equation}
then there exists unique $a^{*}\in (0,\widetilde{a}^{*})$, such that $M(a^{*})=0$, if and only if $K < \beta$, or $K \geq \beta$ and $0\leq\alpha<\alpha_{0}$.
\end{lemma}

\begin{proof}
We take the first derivative of $M(x)$,
\begin{equation}
M'(x)=\frac{(p_{1}-k_{1})\rho(\rho+\lambda_{2}-\mu)+k_{1}\mu\lambda_{2}(1-p_{1})}{(1-k_{1})\rho(\rho+\lambda_{2})(\rho+\lambda_{2}-\mu)}\beta (\widetilde{a}^{*})^{-p_2}p_{2}x^{p_2-1}+\frac{p_{1}-1}{\rho+\lambda_{2}-\mu}.
\end{equation}

Conditions $\mu>0$, $p_1<k_1<0$ and $p_2>1$ indicate $M'(x)<0$ for all $x>0$ and thus $M(x)$ is a decreasing function in $x$. $M(0)=-\frac{p_{1}K}{\rho+\lambda_{2}} > 0$ and $\lim_{x\to\infty}M(x)=-\infty$ imply that there exists unique $a^{*}$ such that $M(a^{*})=0$.

Now we need to check whether $\widetilde{a}^{*}>a^{*}$ or not. Note $\widetilde{a}^{*}>a^{*}$ if and only if $M(\widetilde{a}^{*})<0$.
\begin{align}
M(\widetilde{a}^{*})=&\frac{(p_{1}-k_{1})\rho(\rho+\lambda_{2}-\mu)+k_{1}\mu\lambda_{2}(1-p_{1})}{(1-k_{1})\rho(\rho+\lambda_{2})(\rho+\lambda_{2}-\mu)}\beta
+\frac{p_{1}-1}{\rho+\lambda_{2}-\mu}\widetilde{a}^{*}-\frac{p_{1}K}{\rho+\lambda_{2}} \nonumber \\
=&\frac{(p_{1}-k_{1})\rho(\rho+\lambda_{2}-\mu)+k_{1}\mu\lambda_{2}(1-p_{1})}{(1-k_{1})\rho(\rho+\lambda_{2})(\rho+\lambda_{2}-\mu)}\beta+\frac{p_{1}-1}{\rho+\lambda_{2}
-\mu}\frac{\rho-\mu}{\rho}\frac{k_{1}}{k_{1}-1}\frac{\beta}{\alpha}-\frac{p_{1}K}{\rho+\lambda_{2}}.
\end{align}

It's easy to check that $M(\widetilde{a}^{*})=H(\widetilde{a}^{*})$ (in the proof of Lemma \ref{lemma 1}). Follow the last part of the proof in Lemma \ref{lemma 1}, we immediately obtain the result of Lemma \ref{lemma 2}.
\end{proof}

Now let's focus on the proof of Proposition \ref{Proposition 2}.

\begin{proof}
We consider a candidate stopping time $\tau_{a}=\inf\{t\geq 0: X_{t}\leq a^{*}\}$, where $a^{*}>0$. We split the problem into two cases: $a^{*}\geq\widetilde{a}^{*}$ and $a^{*}<\widetilde{a}^{*}$.

(I) We first assume $a^{*}\geq\widetilde{a}^{*}$. In the stopping region $\{x<a^{*}\}$, $V(x) = 0$ from the definition $\tau_{a}$ and \eqref{expression_V}. In the continuation region $\{x\geq a^{*}\}$, $V(x)$ satisfies the following ODE:
\begin{equation}
\label{ODE V 1}
-(\rho+\lambda_{2})V(x)+\mu xV'(x)+\frac{\sigma^{2}x^{2}}{2}V''(x)+f(x)+\lambda_{2}\widetilde{V}(x)=0,
\end{equation}
with boundary conditions:
\begin{align}
\label{boundary condition V1}
V(a^{*})=0,  \,\,\,\,\,\, V'(a^{*})=0, \,\,\,\,\,\, \lim_{x\to\infty}V'(x)<+\infty.
\end{align}

The ODE and last condition in \eqref{boundary condition V1} indicate that the general solution is given by $C_{1}x^{p_{1}}+v_{1}(x)$, where $C_{1}$ is a constant, $p_{1}$ is the negative root of the associated  quadratic equation, and $v_{1}(x)$ is a  particular solution of \eqref{ODE V 1}:
\begin{equation}
\label{v1(x)}
v_{1}(x)=-\frac{\alpha}{k_{1}(\rho-\mu)}(\widetilde{a}^{*})^{1-k_{1}}x^{k_{1}}+\frac{\rho+\alpha \lambda_{2}-\mu}{\rho+\lambda_{2}-\mu}\frac{1}{\rho-\mu}x-\frac{\beta \lambda_{2}+\rho K}{\rho(\rho+\lambda_{2})}.
\end{equation}

To determine $a^{*}$ and $C_{1}$, we utilize boundary conditions \eqref{boundary condition V1} to get the following system of equations:
\begin{align}
\label{system equations V case1 1}
C_{1}(a^{*})^{p_{1}}+v_{1}(a^{*})&=0, \\
\label{system equations V case1 2}
C_{1}p_{1}(a^{*})^{p_{1}-1}+v_{1}'(a^{*})&=0.
\end{align}
The above equations yield \eqref{a1}. From Lemma \ref{lemma 1}, if $K \geq \beta$ and $\alpha_{0}\leq\alpha\leq 1$, then there exists unique $a^{*}\geq\widetilde{a}^{*}$,
satisfying \eqref{a1}. We solve $a^{*}$ numerically and plug it in \eqref{system equations V case1 1} and \eqref{system equations V case1 2}, then we can get \eqref{C1} and thus \eqref{solution V 1}.

(II) Next we assume $a^{*}<\widetilde{a}^{*}$. Also from the definition $\tau_{a}$ and \eqref{expression_V}, $V(x) = 0$ in the stopping region $\{x<a^{*}\}$. In the continuation region $\{x\geq a^{*}\}$, $V(x)$ satisfies the following ODE:
\begin{equation}
-(\rho+\lambda_{2})V(x)+\mu xV'(x)+\frac{\sigma^{2}x^{2}}{2}V''(x)+f(x)+\lambda_{2}\widetilde{V}(x)=0,
\end{equation}
with continuous fit and smooth-pasting conditions:
\begin{align}
V^{(2)}(a^{*})&=0,  \,\,\,\,\,\,
V'^{(2)}(a^{*})=0, \,\,\,\,\,\,
V^{(1)}(\widetilde{a}^{*})=V^{(2)}(\widetilde{a}^{*}),  \\
\label{boundary condition V2}
V'^{(1)}(\widetilde{a}^{*})&=V'^{(2)}(\widetilde{a}^{*}),  \,\,\,\,\,\,
\lim_{x\to\infty}V'^{(1)}(x)<+\infty.
\end{align}

Since $\widetilde{V}(x)$ is a piecewise function, the general solution is also a piecewise function,
\begin{equation}
V(x)=
\begin{cases}
V^{(1)}(x) &\mbox{if $x>\widetilde{a}^{*}$,}\\
V^{(2)}(x) &\mbox{if $a^{*}<x<\widetilde{a}^{*}$,}\\
0 &\mbox{if $0<x<{a}^{*}$,}\\
\end{cases}
\end{equation}
where
\begin{align}
V^{(1)}(x) &= C_{2}x^{p_{1}}+v_{1}(x), \label{VV1}\\
V^{(2)}(x) &= C_{3}x^{p_{1}}+C_{4}x^{p_{2}}+v_{2}(x).\label{VV2}
\end{align}
The constants $C_{2}$ to $C_{4}$ are     determined by \eqref{boundary condition V2}, and  $p_{1}$ and $p_{2}$ are the two roots of  the quadratic equation: $-(\rho+\lambda_{2})+\mu p+\sigma^{2}p(p-1)/2=0$. In \eqref{VV1} and \eqref{VV2},   the particular solution $v_{1}(x)$ is given by \eqref{v1(x)}, and
\begin{equation}
v_{2}(x)=\frac{1}{\rho+\lambda_{2}-\mu}x-\frac{K}{\rho+\lambda_{2}}.
\end{equation}
To determine $a^{*}$,  $C_{2}$, $C_3$, and $C_{4}$, we refer to \eqref{boundary condition V2} to get the following system of equations:
\begin{align}
\label{system equations V case2 1}
C_{3}(a^{*})^{p_{1}}+C_{4}(a^{*})^{p_{2}}+v_{2}(a^{*})&=0, \\
\label{system equations V case2 2}
C_{3}p_{1}(a^{*})^{p_{1}-1}+C_{4}p_{2}(a^{*})^{p_{2}-1}+v_{2}'(a^{*})&=0, \\
\label{system equations V case2 3}
C_{3}(\widetilde{a}^{*})^{p_{1}}+C_{4}(\widetilde{a}^{*})^{p_{2}}+v_{2}(\widetilde{a}^{*})&=C_{2}(\widetilde{a}^{*})^{p_{1}}+v_{1}(\widetilde{a}^{*}), \\
\label{system equations V case2 4}
C_{3}p_{1}(\widetilde{a}^{*})^{p_{1}-1}+C_{4}p_{2}(\widetilde{a}^{*})^{p_{2}-1}+v_{2}'(\widetilde{a}^{*})&=C_{2}p_{1}(\widetilde{a}^{*})^{p_{1}-1}+v_{1}'(\widetilde{a}^{*}).
\end{align}

Solving the above equations, we obtain $C_2$ to $C_4$ as in \eqref{C3}-\eqref{C6}, and $a^*$ satisfies \eqref{a2}. From Lemma \ref{lemma 2}, if $K < \beta$, or $K \geq \beta$ and $0\leq\alpha<\alpha_{0}$, then there exists a unique  $a^{*}<\widetilde{a}^{*}$
that satisfies  \eqref{a2}.   \end{proof}

\subsection{Proof of Corollary \ref{corollary 3}}
\begin{proof} First, recall that  $k_1, p_1<0$ and $\rho>\mu$, which after some algebra implies that  $0<\alpha_0\leq 1$ (see \eqref{alpha0}).  We rewrite $\alpha_0$ in  the following form
\begin{align}
\alpha_{0}&=\frac{k_{1}(p_{1}-1)(\rho-\mu)(\rho+\lambda_{2})}{k_{1}(p_{1}-1)(\rho-\mu)(\rho+\lambda_{2})-\rho p_{1}(k_{1}-1)(\rho+\lambda_{2}-\mu)(1-\frac{K}{\beta})} \\
&=\frac{\lambda_2S_1(p_1)+S_2(p_1)}{\lambda_2S_3(p_1)+S_4(p_1)},
\end{align}
where
\begin{equation}
\label{equations_of_p1}
\begin{array}{ll}
S_1(p_1)&=k_1(\rho-\mu)(p_1-1),\\
S_2(p_1)&=\rho S_1(p_1),\\
\end{array}
\hspace{2em}
\begin{array}{ll}
S_3(p_1)&=S_1(p_1)-\rho(k_1-1)(1-\frac{K}{\beta})p_1,\\
S_4(p_1)&=S_2(p_1)+(\rho-\mu)(S_3(p_1)-S_1(p_1)),\\
\end{array}
\end{equation}
and $\lambda_2S_3(p_1)+S_4(p_1)>0$ under the condition $K>\beta$. Since $p_1$ depends on $\lambda_2$, we differentiate to  get
\begin{equation}
\label{derivatives_of_p1}
\begin{array}{ll}
\frac{\partial S_1(p_1)}{\partial\lambda_2}&=k_1(\rho-\mu)\frac{\partial p_1}{\lambda_2},\\
\\
\frac{\partial S_2(p_1)}{\partial\lambda_2}&=\rho\frac{\partial S_1(p_1)}{\partial\lambda_2},\\
\end{array}
\begin{array}{ll}
\frac{\partial S_3(p_1)}{\partial\lambda_2}&=\frac{\partial S_1(p_1)}{\partial\lambda_2}-\rho(k_1-1)(1-\frac{K}{\beta})\frac{\partial p_1}{\lambda_2},\\
\\
\frac{\partial S_4(p_1)}{\partial\lambda_2}&=\frac{\partial S_2(p_1)}{\partial\lambda_2}+(\rho-\mu)(\frac{\partial S_3(p_1)}{\partial\lambda_2}-\frac{\partial S_1(p_1)}{\partial\lambda_2}),\\
\end{array}
\end{equation}
where
\begin{equation}
\label{p1_with_respect_to_lambda2}
\frac{\partial p_1}{\lambda_2}=-\left((\frac{\sigma^{2}}{2}-\mu)^{2}+2\sigma^{2}(\rho+\lambda_{2})\right)^{-\frac{1}{2}} <0.
\end{equation}

In turn,  the partial derivative of $\alpha_0$ in $\lambda_2$ is given by
\begin{equation}
\frac{\partial\alpha_0}{\partial\lambda_2}=\left((\lambda_2\frac{\partial S_1}{\partial\lambda_2}+S_1+\frac{\partial S_2}{\partial\lambda_2})(\lambda_2S_3+S_4)-(\lambda_2\frac{\partial S_3}{\partial\lambda_2}+S_3+\frac{\partial S_4}{\partial\lambda_2})(\lambda_2S_1+S_2)\right)\frac{1}{\left(\lambda_2S_3+S_4\right)^{2}},
\end{equation}where $S_i \equiv S_i(p_1)$, for $1\leq i \leq 4$.
Since $(\lambda_2S_3+S_4)^{2}>0$, we omit last term in the following calculation,
\begin{equation}
\frac{\partial\alpha_0}{\partial\lambda_2}=(\frac{\partial S_1}{\partial\lambda_2}S_3-\frac{\partial S_3}{\partial\lambda_2}S_1)\lambda^{2}_{2}+(\frac{\partial S_1}{\partial\lambda_2}S_4-\frac{\partial S_4}{\partial\lambda_2}S_1+\frac{\partial S_2}{\partial\lambda_2}S_3-\frac{\partial S_3}{\partial\lambda_2}S_2)\lambda_2+\frac{\partial S_2}{\partial\lambda_2}S_4-\frac{\partial S_4}{\partial\lambda_2}S_2+S_1S_4-S_2S_3.
\end{equation}
With \eqref{equations_of_p1} and \eqref{derivatives_of_p1}, we get the final result,
\begin{equation}
\label{alpha0_derivative}
\frac{\partial\alpha_0}{\partial\lambda_2}=(1-\frac{K}{\beta})k_1\rho(1-k_1)(\rho-\mu)\left((\lambda_2+\rho)(\lambda_2+\rho-\mu)\frac{\partial p_1}{\lambda_2}+p_1\mu(1-p_1)\right).
\end{equation}
From \eqref{p1_with_respect_to_lambda2}, we see that  $\frac{\partial\alpha_0}{\partial\lambda_2}$ in \eqref{alpha0_derivative} is strictly negative for all $\lambda_2>0$ under the condition $K>\beta$.

Since $\alpha_0$ is decreasing in $\lambda_2$ and bounded below by  zero, $\alpha_0$ must converge. Notice that $S_i(p_1) (1\leq i \leq 4)$ are all linear function of $p_1$. Hence, we have
\begin{align}
\lim_{\lambda_{2}\to+\infty}\alpha_{0}
&=\lim_{\lambda_{2}\to+\infty}\frac{\lambda_2S_1(p_1)+S_2(p_1)}{\lambda_2S_3(p_1)+S_4(p_1)}\\
&=\lim_{p_1\to-\infty}\left(1+\frac{p_1\rho(1-k_1)}{(p_1-1)k_1(\rho-\mu)}(1-\frac{K}{\beta})\right)^{-1}\\
&=\left(1+\frac{\rho(1-k_1)}{k_1(\rho-\mu)}(1-\frac{K}{\beta})\right)^{-1}.
\end{align}
This validates \eqref{alpha0_limit}, and  $0<\alpha_0<1$ implies that  $0<\alpha^{\infty}_0\leq 1$. The equality $\alpha^{\infty}_0 =1$  iff $K=\beta$ by \eqref{alpha0_limit}.\end{proof}

\subsection{Proof of Proposition \ref{Proposition 3}}
\begin{proof}
We consider two candidate stopping times $\tau_{e}=\inf\{t\geq 0: X_{t}\geq e^{*}\}$ and
$\tau_{c}=\inf\{t\geq 0: X_{t}\leq c^{*}\}$, where $e^{*}\ge c^{*}>0$. In the   entry region $\{x>e^{*}\}$, the firm enters the market and thus $\psi(x)=V(x)$. In the cancellation region $\{x<c^{*}\}$, we have $\psi(x)=0$. In the  continuation region, $c^{*}\leq x\leq e^{*}$, where the firm waits for entry or cancellation, the firm's value function $\psi(x)$ satisfies the  ODE:
\begin{equation}
\label{ODE psi}
-(\rho+\lambda_{1})\psi(x)+\mu x\psi'(x)+\frac{\sigma^{2}x^{2}}{2}\psi''(x)-c(x)=0,
\end{equation}
with the boundary conditions:
\begin{align}
\label{boundary condition psi}
\psi(c^{*})=0, \,\,\,\,\,\,
\psi'(c^{*})=0, \,\,\,\,\,\,
\psi(e^{*})=V(e^{*}), \,\,\,\,\,\,
\psi'(e^{*})=V'(e^{*}).
\end{align}
The general solution of the above ODE is given by $\psi(x) = D_{1}x^{q_{1}}+D_{2}x^{q_{2}}+c_{0}(x)$, where  $q_{1}$ and $q_{2}$ are roots of  the quadratic equation: $-(\rho+\lambda_{1})+\mu q+{\sigma^{2}}q(q-1)/2=0$,
and $c_{0}(x)$ is a linear particular solution of \eqref{ODE psi}:
\begin{equation}
c_{0}(x)=-\frac{a}{\rho+\lambda_{1}-\mu}x-\frac{b}{\rho+\lambda_{1}}.
\end{equation} To determine $c^{*}$ and $e^{*}$, along with the constants $D_{1}$ and $D_{2}$, we  apply the conditions in  \eqref{boundary condition psi} to obtain the system of equations \eqref{ce1}-\eqref{ce4}.
\end{proof}

 \bibliographystyle{apa}
\bibliography{paper}
 
\end{document}